\documentclass[11pt]{amsart}

\usepackage{geometry}

\usepackage{multicol}
\usepackage{mathtools}

\usepackage{amsfonts, amssymb, amscd}
\numberwithin{equation}{section}

\usepackage[symbol]{footmisc}


\usepackage{bm}
\usepackage{verbatim}
\usepackage{amssymb}
\usepackage{mathrsfs}
\usepackage{graphicx}
\usepackage[noadjust]{cite}
\usepackage[all]{xy}
\usepackage{tikz-cd}
\usepackage{subcaption}
\usepackage{subfiles}
\usepackage[toc,page]{appendix}
\usepackage{mathtools}
\usepackage{comment}
\usepackage{enumerate}
\usepackage{tikz}
\usepackage{marvosym}
\usepackage{bm}

\usepackage{graphicx}
\graphicspath{{images/}}

\usepackage{appendix}
\usepackage{hyperref}

\newcommand{\KK}{\mathbb{K}}

\newcommand{\Qq}{\mathbb{Q}}

\newcommand{\Rr}{\mathbb{R}}

\newcommand{\Zz}{\mathbb{Z}}

\newcommand{\fD}{\mathfrak{D}}

\newcommand{\Center}{\operatorname{center}}

\newcommand{\Exc}{\operatorname{Exc}}

\newcommand{\mld}{{\rm{mld}}}

\newcommand{\lct}{\operatorname{lct}}

\newcommand{\Supp}{\operatorname{Supp}}
\newcommand{\rank}{\operatorname{rank}}

\newcommand{\mult}{\operatorname{mult}}

\newcommand{\Oo}{\mathcal{O}}

\newcommand{\mf}{\mathcal{MF}}
\newcommand{\nmul}{\mathsf{nm}}
\newcommand{\ndist}{\mathsf{nd}}
\newcommand{\NP}{\mathcal{NP}}

\newcommand\bmu{{\bm \mu}}

\newtheorem{thm}{Theorem}[section]
\newtheorem{conj}[thm]{Conjecture}
\newtheorem{cor}[thm]{Corollary}
\newtheorem{lem}[thm]{Lemma}
\newtheorem{prop}[thm]{Proposition}

\newtheorem{claim}[thm]{Claim}

\newtheorem*{notation}{Notation ($\star$)}

\theoremstyle{definition}
\newtheorem{rem}[thm]{Remark}

\newtheorem{ex}[thm]{Example}

\newtheorem{defn}[thm]{Definition}
\newtheorem{question}[thm]{Question}

\usepackage{todonotes}

\newcommand\luo[1]{\todo[color=green!40]{#1}}

\newcommand\JC[1]{\todo[color=blue!40]{#1}}

\title{Shokurov's conjecture on conic bundles with canonical singularities}

\author{Jingjun Han}
\address{Jingjun Han, Department of Mathematics, Johns Hopkins University, Baltimore, MD 21218, USA}
\email{jhan@math.jhu.edu}

\author{Chen Jiang}
\address{Chen Jiang, Shanghai Center for Mathematical Sciences, Fudan University, Jiangwan Campus, Shanghai, 200438, China}
\email{chenjiang@fudan.edu.cn}

\author{Yujie Luo}
\address{Yujie Luo, Department of Mathematics, Johns Hopkins University, Baltimore, MD 21218, USA}
\email{yluo32@jhu.edu}

\begin{document}
\maketitle

\begin{abstract}
A conic bundle is a contraction $X\to Z$ between normal varieties of relative dimension $1$ such that $-K_X$ is relatively ample. We prove a conjecture of Shokurov which predicts that, if $X\to Z$ is a conic bundle such that $X$ has canonical singularities and $Z$ is $\Qq$-Gorenstein, then $Z$ is always $\frac{1}{2}$-lc, and the multiplicities of the fibers over codimension $1$ points are bounded from above by $2$. Both values $\frac{1}{2}$ and $2$ are sharp. This is achieved by solving a more general conjecture of Shokurov on singularities of bases of lc-trivial fibrations of relative dimension $1$ with canonical singularities.
\end{abstract}

\tableofcontents

\section{Introduction}

We work over the field of complex numbers $\mathbb{C}$.

A {\it $\Qq$-conic bundle} is a proper morphism $X\to Z$ from a $3$-fold with only terminal singularities to a normal surface such that all fibers are connected and 1-dimensional, and $-K_X$ is relatively ample over $Z$.
A conjecture of Iskovskikh predicts that the base surface $Z$ has only canonical singularities, or equivalently $Z$ is $1$-lc. This conjecture has important applications to the rationality problem of conic bundles \cite{Isk96}. Mori and Prokhorov proved Iskovskikh's conjecture by showing that $Z$ has only Du Val singularities of type $A$ and giving a complete local classification of $\Qq$-conic bundles over a singular base in \cite{MP08, MP08b}.

Motivated by Iskovskikh's conjecture, it is natural to study the singularities of the base surface $Z$ when $X$ has worse singularities, for example, canonical singularities. Such kind of contraction also appears naturally in the birational classification of $3$-dimensional algebraic varieties. Indeed when $\rho(X/Z)=1$, it is one of three possible outcomes of the minimal model program for canonical $3$-folds of negative Kodaira dimension. However, $Z$ may no longer be $1$-lc for such contractions. Shokurov conjectured that $Z$ is always $\frac{1}{2}$-lc, and the value $\frac{1}{2}$ is optimal (see Remark~\ref{remark 1}). More generally, Shokurov's conjecture is expected to hold for conic bundles with canonical singularities in all dimensions. 

\begin{conj}[{Shokurov, cf. \cite{Sho14, Pro18}}]\label{con: Shokurov P1 fibration lct 1/2}
Let $\pi:X\to Z$ be a contraction between normal varieties, such that
\begin{enumerate}
 \item $\dim X-\dim Z=1$, 
 \item $X$ is canonical,
 \item $K_Z$ is $\mathbb{Q}$-Cartier, and
 \item\label{con: Shokurov P1 assum 4} $-K_X$ is ample over $Z$.
\end{enumerate}
Then $Z$ is $\frac{1}{2}$-lc.
\end{conj}

\begin{rem}\label{remark 1}
\begin{enumerate}
 \item In Conjecture~\ref{con: Shokurov P1 fibration lct 1/2}, assumption~\eqref{con: Shokurov P1 assum 4} can be replaced by ``$-K_X$ is nef and big over $Z$", which can be reduced to Conjecture~\ref{con: Shokurov P1 fibration lct 1/2} by taking the anti-canonical model over $Z$.
\item In a private communication, Prokhorov shared his expectation that $Z$ should be $\frac{1}{2}$-klt in Conjecture~\ref{con: Shokurov P1 fibration lct 1/2} motivated by \cite[Example~10.6.1]{Pro18}. However this is not always the case if $\dim X\geq 3$, see Example~\ref{ex: counterex of 1/2-klt}. 
\end{enumerate}
\end{rem}

\begin{ex}[{cf. \cite[Example~10.6.1]{Pro18}}]\label{ex: counterex of 1/2-klt}
Consider the following action of $\bmu_{4m}$ on $\mathbb{P}_x^1\times \mathbb{C}^2_{u, v}$:
$$
(x; u, v)\mapsto (-x; \xi u, \xi^{2m-1} v),
$$
where $m$ is a positive integer and $\xi$ is a primitive $4m$-th root of unity. Let $X=(\mathbb{P}^1\times \mathbb{C}^2)/\bmu_{4m}$, $Z= \mathbb{C}^2/\bmu_{4m}$, and $\pi: X\to Z$ the natural projection. Since $\bmu_{4m}$ acts freely in codimension $1$, $-K_X$ is $\pi$-ample. Note that $Z$ has an isolated cyclic quotient singularity of type $\frac{1}{4m}(1, 2m-1)$ at the origin $o\in Z$, and $\mld(Z\ni o)=\frac{1}{2}$ (see \cite{Amb06} for the computation of minimal log discrepancies of toric varieties). On the other hand, $X$ is covered by $2$ open affine charts $(x\neq 0)$ and $(x\neq \infty)$, and each chart is isomorphic to the affine toric variety $\mathbb{C}^3/\frac{1}{4m}(2m, 1, 2m-1)$, which is canonical (see \cite[(4.11)~Theorem]{YPG}) and Gorenstein. Note that in this case, $\rho(X/Z)=1$ and the singular locus of $X$ is the whole fiber $\pi^{-1}(o)$ which is $1$-dimensional. It is not clear yet whether there are such examples where $X$ has isolated canonical singularities.
\end{ex}

The main purpose of this paper is to give an affirmative answer to Shokurov's conjecture.
\begin{thm}\label{thm: Shokurov P1 fibration lct 1/2}
Conjecture~\ref{con: Shokurov P1 fibration lct 1/2} holds.
\end{thm}

Theorem~\ref{thm: Shokurov P1 fibration lct 1/2} follows from a more general result, see Theorem~\ref{High dim SM}. In order to state the result, we recall some backgrounds. Let $\pi:(X,B)\to Z$ be an lc-trivial fibration (see Definition~\ref{def lctrivial}, for example,
$\pi: X\to Z$ is a contraction between normal varieties and
$(X, B)$ is an lc pair with $K_X+B\sim_{\Rr,Z}0$). 
By the work of Kawamata \cite{Kaw97,Kaw98} and Ambro \cite{Amb05}, we have the so-called {\it canonical bundle formula}
$$K_X+B\sim_{\Rr}\pi^{*}(K_Z+B_Z+M_Z),$$
where $B_Z$ is the \emph{discriminant part} and $M_Z$ is the \emph{moduli part}, see Section~\ref{sec cbf} for more details. For the inductive purpose, it is useful and important to study the relation between singularities of $(X,B)$ and those of $(Z,B_Z+M_Z)$. In this context, Shokurov proposed the following conjecture. Recall that $\mld(X/Z\ni z,B)$ is the infimum of all the log discrepancies of prime divisors over $X$ whose image on $Z$ is $\overline{z}$ (see Definition~\ref{defn: relative mld}).

\begin{conj}[Shokurov, cf. {\cite[Conjecture~1.2]{AB14}}]\label{conj:ShoMc}
	Let $d$ be a positive integer and $\epsilon$ a positive real number. Then there is a positive real number $\delta=\delta(d,\epsilon)$ depending only on $d,\epsilon$ satisfying the following. Let $\pi: (X,B)\to Z$ be an lc-trivial fibration and $z\in Z$ a point of codimension $\ge 1$, such that 
	\begin{enumerate}
	\item $\dim X-\dim Z=d$,
	 \item\label{conj:ShoMc assum 2} $\mld(X/Z\ni z,B)\ge \epsilon$, and 
	 \item the generic fiber of $\pi$ is of Fano type.
	\end{enumerate}
Then we can choose $M_Z\geq 0$ representing the moduli part, so that $(Z\ni z,B_Z+M_Z)$ is $\delta$-lc. 
\end{conj}

\begin{rem}
\begin{enumerate}
 \item The formulation of Conjecture~\ref{conj:ShoMc} here is stronger than that in the previous literature \cite{AB14, Bir16crelle}, where a stronger assumption (2') that ``$(X, B)$ is an $\epsilon$-lc pair" is required instead of assumption~\eqref{conj:ShoMc assum 2}, and $\delta$ depends on $\dim X$ and $\epsilon$ instead of just $\dim X-\dim Z$ and $\epsilon$.
 In our formulation, $B$ can be non-effective and $(X, B)$ can have non-klt centers over $Z\setminus z$. 
 
\item Birkar \cite{Bir16crelle} proved Conjecture~\ref{conj:ShoMc} under assumption (2') for one of the following cases: (a) $(F,B|_F)$ belongs to a bounded family, or (b) $\dim X=\dim Z+1$. Hence Conjecture~\ref{conj:ShoMc} under assumption (2') holds when the coefficients of $B|_F$ are bounded from below away from zero as a consequence of the Borisov--Alexeev--Borisov conjecture proved by Birkar \cite{Bir19,Bir21}. Very recently, Birkar and Y. Chen \cite{BC21} proved Conjecture~\ref{conj:ShoMc} under assumption (2') for toric morphisms between toric varieties. We refer the readers to \cite[Theorems~1.9 and 2.5]{Bir18} for more related results.

\item Following ideas in \cite{Bir16crelle}, it is indicated by G. Chen and the first author \cite[Proposition~7.6]{CH21} that Conjecture~\ref{conj:ShoMc} might be a consequence of Shokurov's $\epsilon$-lc complements conjecture. Moreover, following the proof of \cite[Corollary~1.7]{Bir16crelle}, \cite[Theorem~1.3]{CH21} implies that Conjecture~\ref{conj:ShoMc} holds for $\dim X=\dim Z+1$.

\item It is worthwhile to mention that Conjecture~\ref{conj:ShoMc} implies M\textsuperscript{c}Kernan's conjecture on Mori fiber spaces \cite[Conjecture~1.1]{AB14}, which is closely related to Iskovskikh's conjecture. Alexeev and Borisov \cite{AB14} proved M\textsuperscript{c}Kernan's conjecture for toric morphisms between toric varieties. 
 
\end{enumerate}
\end{rem}

Our second main result gives the optimal value of $\delta(1,\epsilon)=\epsilon-\frac{1}{2}$ for any $\epsilon\ge 1$.

\begin{thm}\label{High dim SM}
Let $\pi: (X,B)\to Z$ be an lc-trivial fibration and $z\in Z$ a codimension $\ge 1$ point, such that

\begin{enumerate}
 \item $\dim X-\dim Z=1$,
 \item $\mld(X/Z\ni z,B)\ge 1$, and 
 \item the generic fiber of $\pi$ is a rational curve.
\end{enumerate}
Then we can choose $M_Z\geq 0$ representing the moduli part, so that 
$$\mld(Z\ni z,B_Z+M_Z)\geq \mld(X/Z\ni z,B)-\frac{1}{2}\geq \frac{1}{2}.$$
\end{thm}

The lower bound in Theorem~\ref{High dim SM} is optimal by Example~\ref{ex:lct is attained}.

As a corollary, we have the following global version of Theorem~\ref{High dim SM} with less technical notation involved.
\begin{cor}\label{cor: High dim SM}
Let $(X,B)$ be a pair, and $\pi: X\to Z$ a contraction between normal varieties such that

\begin{enumerate}
 \item $\dim X-\dim Z=1$,
 \item\label{thm High dim SM assum 2} $(X,B)$ is canonical and $B$ has no vertical irreducible component over $Z$,
 \item $K_X+B\sim_{\Rr,Z} 0$, and
 \item $X$ is of Fano type over $Z$.
\end{enumerate}
Then we can choose $M_Z\geq 0$ representing the moduli part, so that $(Z,B_Z+M_Z)$ is $\frac{1}{2}$-lc.
\end{cor}

\begin{rem}
\begin{enumerate}
\item We remark that if $\dim X-\dim Z=1$, 
then assumption~\eqref{thm High dim SM assum 2} in Corollary~\ref{cor: High dim SM} is equivalent to the assumption that $\mld(X/Z\ni z, B)\geq 1$ for any codimension $\geq 1$ point $z\in Z$. 

\item Note that $\frac{1}{2}$ is the maximal accumulation point of 
the set of minimal log discrepancies in dimension $2$ (see \cite[Corollary~3.4]{Ale93}, \cite{Sho94}). 
Thus it would be interesting if one could give a new proof of Iskovskikh's conjecture by applying Theorems~\ref{thm: Shokurov P1 fibration lct 1/2} and~\ref{High dim SM} without using the classification of terminal singularities in dimension $3$. In fact, we can apply Corollary~\ref{cor: High dim SM} to show that in the setting of Iskovskikh's conjecture, $Z$ is $\frac{1}{2}$-klt, see Corollary~\ref{cor weak isk}. 
Recall that in order to prove Iskovskikh's conjecture, it suffices to show that $Z$ is $\frac{2}{3}$-klt (see \cite[Lemma~5.1]{Jia19}), but our method could not archive this. The reason is that in Corollary~\ref{cor weak isk}, there is no assumption on $\dim X$, but Prokhorov provides us Example~\ref{ex: counterex 4.8} showing that Corollary~\ref{cor weak isk} can not be improved if $\dim X\geq 4.$ 
\end{enumerate}
\end{rem}

Theorem~\ref{High dim SM} is a consequence of the following result which gives a lower bound of certain log canonical thresholds for lc-trivial fibrations. We refer the readers to \cite[Problem~7.18]{CH21} for more discussions.

\begin{thm}[cf. {\cite[Conjecture]{Sho14}}] \label{High dim}
Let $\pi: (X,B)\to Z$ be an lc-trivial fibration and $z\in Z$ a codimension $1$ point, such that
\begin{enumerate}
\item $\dim X-\dim Z=1$,
 \item $\mld(X/Z\ni z,B)\geq 1$, and
\item the generic fiber of $\pi$ is a rational curve.
\end{enumerate}
Then 
$$\lct(X/Z\ni z, B; \pi^*\overline{z})\geq \mld(X/Z\ni z,B)-\frac{1}{2}\geq \frac{1}{2}.$$ In particular, if $B$ is effective, then the multiplicity of each irreducible component of $\pi^*{z}$ is
bounded from above by $2$.
\end{thm} 

The bounds in Theorem~\ref{High dim} are optimal by Example~\ref{ex:lct is attained}. 

Y. Chen informed us that together with Birkar, they also got the lower bound $\frac{1}{2}$ in Theorem~\ref{High dim} for toric morphisms between toric varieties in an earlier version of \cite{ BC21}. As a related result, when $\dim X-\dim Z=2$, Mori and Prokhorov \cite{MP09} showed that any 3-dimensional terminal del Pezzo fibration has no fibers of multiplicity $>6$.

It turns out that Theorem~\ref{High dim} can be reduced to a local problem on estimating the lower bound of the log canonical threshold of a smooth curve with respect to a canonical pair on a smooth surface germ, see Corollary~\ref{Smooth local case}. We prove a general result here as it might have broader applications in other topics in birational geometry (cf. \cite[Corollary~6.46]{KSC04}).

\begin{thm}\label{Smooth local case general beta version}
Let $(X\ni P, B)$ be a germ of surface pair such that $X$ is smooth and $\mult_P B\leq 1$. Let $C$ be a smooth curve at $P$ such that $C\nsubseteq\Supp B$.
Denote $\mult_P B=m$, $(B\cdot C)_P=I$. Then $\lct(X\ni P,B;C)\geq \min\{1, 1+\frac{m}{I}-m\}$. 
\end{thm}

Example~\ref{ex optimal lower bound} shows that the lower bound in Theorem~\ref{Smooth local case general beta version} is optimal (even in the case when $\Supp B$ is irreducible). It would be interesting to get an optimal lower bound of $\lct(X\ni P,B;C)$ if we do not assume that $C$ is smooth in Theorem~\ref{Smooth local case general beta version} as it might be related to alpha invariants.

\medskip

It would also be interesting to ask the following question.

\begin{question}
When $\dim X=3$,
can one give a complete local classification of the extremal case in Conjecture~\ref{con: Shokurov P1 fibration lct 1/2} when $Z$ is strictly $\frac{1}{2}$-lc? Or more generally, can one give a complete local classification in Conjecture~\ref{con: Shokurov P1 fibration lct 1/2} when $Z$ is singular?
\end{question}

\noindent\textbf{Sketch of proofs.} 
By applying \cite[Theorem 8.1]{PS09}, we may reduce Theorem~\ref{High dim SM} to Theorem~\ref{High dim}. Here the sub-pair setting plays a key role, which makes this reduction step simpler than that of the pair setting (cf. \cite[Lemma 3.4, Proposition 3.5]{Bir16crelle}), and it enables us to treat the case $\mld(X/Z\ni z,B)>1$. On the other hand, the sub-pair setting causes new technical difficulties in the proof of Theorem~\ref{High dim}. 
By taking hyperplane sections of the base $Z$ we may reduce Theorem~\ref{High dim} to the case $\dim X=2$. By an MMP argument, we may 
 reduce Theorem~\ref{High dim} to the case when $X\to Z$ is a $\mathbb{P}^1$-bundle and $B\ge0$, so the problem is reduced to a special case of Theorem~\ref{Smooth local case general beta version} when $\mult_P B\leq 1$ and $(B\cdot C)_P\leq 2$. Since the conditions $\mult_P B\leq 1$ and $(B\cdot C)_P\leq 2$ do not behave well under blow-ups, one may encounter difficulties by applying the ideas in \cite{Ale93,CH21,HL20} which deal with the minimal log discrepancies for surfaces. The key idea is that, we consider $\widehat{X}$, the completion of $X$ along $P$, and decompose $B$ into irreducible components on $\widehat{X}$. By using the log canonical threshold polytope and applying the convexity of log canonical thresholds in a careful way, we may reduce Theorem~\ref{Smooth local case general beta version} to the case when $\Supp B$ is irreducible on $\widehat{X}$. Here recall that $\lct(X\ni P, B;C)=\lct(\widehat{X}\ni P, B;C)$. Finally, for this last case, following the ideas in \cite{Kuw99}, we may give a lower bound of $\lct(\widehat{X}\ni P,B;C)$ by using the first pair of Puiseux exponents of $B$. The proof of Theorem~\ref{Smooth local case general beta version} is provided in Section~\ref{sec lct}, and the proofs of other main results in this paper are provided in Section~\ref{sec proofs}. We refer the reader to Appendix~\ref{sec appendix} for a different proof of a weaker version of Theorem~\ref{Smooth local case general beta version} (see Theorem~\ref{Smooth local case general alpha version}) which does not use the convexity and Appendix~\ref{appendix: B second proof} for another proof of Theorem~\ref{Smooth local case general beta version} which does not rely on \cite{Kuw99}.

\medskip

\noindent\textbf{Acknowledgments.} We are grateful to Professor~V.~V.~Shokurov for sharing with us his conjecture (Conjecture~\ref{con: Shokurov P1 fibration lct 1/2}) and for a lot of useful discussions and insightful suggestions. Especially, Professor~Shokurov suggested us to consider sub-pairs in the formulation of main results.
The third author would like to thank his advisor Chenyang Xu for his support. Part of this work was done while the third author visited Zhiyu Tian at BICMR, Peking University during 2020 Fall Semester, and he would like to thank their hospitality. We would like to thank Caucher Birkar, Yifei Chen, Jihao Liu, Shigefumi Mori, and Yuri Prokhorov for helpful comments. The second author was supported by National Key Research and Development Program of China (Grant No. 2020YFA0713200).

\section{Preliminaries}

In this section we collect basic definitions and results. We adopt the standard notation and definitions in \cite{KM98} and \cite{BCHM10}.

\subsection{Divisors}
Let $\KK$ be either the rational number field $\Qq$ or the real number field $\Rr$.
 Let $X$ be a normal variety. A {\it $\KK$-divisor} is a finite $\KK$-linear combination $D=\sum d_{i} D_{i}$ of prime Weil divisors $D_{i}$, and $d_{i}$ denotes the {\it coefficient} of $D_i$ in $D$. A {\it $\KK$-Cartier divisor} is a $\KK$-linear combination of Cartier divisors. 

 We use $\sim_{\KK}$ to denote the $\KK$-linear equivalence between $\KK$-divisors. For a projective morphism $X\to Z$, we use $\sim_{\KK,Z}$ to denote the relative $\KK$-linear equivalence and use $\equiv_{Z}$ to denote the relative numerical equivalence. 

\begin{defn}[cf. \cite{PS09}]\label{defn bdiv}
Let $X$ be a normal variety. Consider an infinite linear combination $\mathbf{D}:=\sum_D d_D D$, where $d_D\in\mathbb{K}$ and the infinite sum runs over all divisorial valuations of the function field of $X$. For any birational model $Y$ of $X$, the \emph{trace} of $\mathbf{D}$ on $Y$ is defined by $\mathbf{D}_Y:=\sum_{\mathrm{codim}_Y D=1} d_D D $. A \emph{b-$\mathbb{K}$-divisor} (or \emph{b-divisor} for short when the base field is clear) is a possibly infinite linear combination of divisorial valuations $\mathbf{D}=\sum_D d_D D$, such that on each birational model $Y$ of $X$, the trace $\mathbf{D}_Y$ is a $\mathbb{K}$-divisor, or equivalently, $\mathbf{D}_Y$ is a finite sum. 
If $d_D\neq 0$ in $\mathbf{D}$ for some $D$, $D$ is called a \emph{birational component} of $\mathbf{D}$.

Let $D$ be a $\mathbb{K}$-Cartier divisor on $X$. The {\it Cartier closure} of $D$ is the b-divisor $\overline{D}$ whose trace on every birational model $f:Y\to X$ is $f^*D$.

A b-divisor $\mathbf{D}$ is said to be
\emph{b-semi-ample} if 
there is a birational model $X'$ over $X$ such that $\mathbf{D}_{X'}$ is $\mathbb{K}$-Cartier and semi-ample, and $\mathbf{D}=\overline{\mathbf{D}_{X'}}$.
\end{defn}

\subsection{Pairs and singularities}\label{the definitions of singularities and pairs}

\begin{defn}\label{defn contraction}
Let $\pi: X\to Z$ be a morphism between varieties.
 We say that $\pi: X \to Z$ is a \emph{contraction} if $\pi$ is projective and $\pi_*\Oo_X = \Oo_Z$. In particular, $\pi$ is surjective and has connected fibers.
\end{defn}

\begin{defn}\label{Horizontal}
Let $\pi: X\to Z$ be a contraction between normal varieties. 
For a prime divisor $E$ on $X$, $E$ is said to be \emph{horizontal} over $Z$ if $E$ dominates $Z$, and $E$ is said to be \emph{vertical} over $Z$ if $E$ does not dominate $Z$.
An $\mathbb{R}$-divisor on $X$ is said to be \emph{vertical} over $Z$ if all its irreducible components are vertical over $Z$.
\end{defn}

\begin{defn}[{cf. \cite[Definition 3.2]{CH21}}] \label{defn sing}
A \emph{sub-pair} $(X, B)$ consists of a normal variety $X$ and an $\mathbb{R}$-divisor $B$ on $X$ such that $K_X + B$ is $\Rr$-Cartier. We say that $(X, B)$ is a \emph{pair} if $(X, B)$ is a sub-pair and $B$ is effective. 

A \emph{(relative) sub-pair} $(X/Z\ni z, B)$ consists of normal varieties $X,Z$, a contraction $\pi: X\to Z$, a scheme-theoretic point $z\in Z$, and an $\mathbb{R}$-divisor $B$ on $X$ such that $K_X + B$ is $\Rr$-Cartier and $\dim z < \dim X$. We say that $(X/Z\ni z, B)$ is a \emph{(relative) pair} if $(X/Z\ni z, B)$ is a sub-pair and $B$ is effective. We say that a pair $(X/Z\ni z, B)$ is a \emph{germ} near $z$ if $z$ is a closed point. 

When $Z=X$, $z=x$, and $\pi$ is the identity map, we will use $(X\ni x, B)$ instead of $(X/Z\ni z, B)$ for simplicity. When $B=0$, we will use $X$ or $X/Z\ni z$
instead of $(X, 0)$ or $(X/Z\ni z, 0)$ for simplicity.
\end{defn}

\begin{defn}\label{defn: relative mld}
 Let $(X/Z\ni z,B)$ be a sub-pair with contraction $\pi: X\to Z$ and $E$ a prime divisor over $X$. Let $\phi :Y\to X$ be a proper birational morphism such that $E$ is a divisor on $Y$ and write $K_Y+B_Y=\phi^{*}(K_X +B)$. The \emph{log discrepancy} of $E$ with respect to $(X, B)$ is defined to be $a(E, X, B):= 1-\mult_EB_Y$, which is independent of the choice of $Y$.
 
Denote
$$\fD(X/Z\ni z):=\{E\mid E\text{ is a prime divisor over }X, \pi(\Center_{X}(E))=\overline{z}\}.$$
The {\it minimal log discrepancy} of $(X/Z\ni z,B)$ is defined to be
$$\mld(X/Z\ni z,B):=\inf\{a(E,X,B) \mid {E\in \fD(X/Z\ni z)}\}.$$

By \cite[Lemma~3.5]{CH21}, the infimum is a minimum if $(X/Z\ni z,B)$ is an lc sub-pair, and it can be computed on a log resolution $\phi: Y\to (X, B)$ where $\Supp(\phi^{-1}(\pi^{-1}(\overline{z})))+\phi_*^{-1}\Supp B+\Exc(\phi)$ is a simple normal crossing divisor. 

When $X=Z$, $z=x$, and $\pi$ is the identity map, we use $\mld(X\ni x,B)$ instead of $\mld(X/Z\ni z,B)$ for simplicity. 
 \end{defn}

\begin{defn}\label{defn: singularities}
Fix a non-negative real number $\epsilon$.
We say that the sub-pair $(X/Z\ni z, B)$ is {\it $\epsilon$-lc} (respectively, {\it $\epsilon$-klt}, {\it klt}, {\it lc}) 
if $\mld(X/Z\ni z,B)\ge \epsilon$ (respectively, $> \epsilon$, $>0$, $\ge0$). 

We say that $(X, B)$ is $\epsilon$-lc (respectively, $\epsilon$-klt, klt, lc) if $(X\ni x, B)$ is so for any codimension $\geq 1$ point $x\in X$; 
we say that $(X, B)$ is \emph{canonical} (respectively, \emph{terminal}) if $a(E,X,B)\ge 1$ (respectively, $a(E,X,B)> 1$) for any exceptional prime divisor $E$ over $X$. These coincide with the usual definitions (cf. \cite[Definition~2.34]{KM98}).
\end{defn}

The following lemma is well-known to experts, which says that being lc over $z\in Z$ is an open condition.
\begin{lem}\label{lem lc near z}
Let $(X/Z\ni z, B)$ be a sub-pair with contraction $\pi: X\to Z$ and fix a log resolution $f:Y\to (X, B)$ such that $f^{-1}\pi^{-1}(\overline{z})$ is a simple normal crossing divisor and write $K_Y+B_Y=f^*(K_X+B)$.
The following are equivalent.
\begin{enumerate}
 \item $(X/Z\ni z, B)$ is lc;
 \item for any prime divisor $E'$ on $Y$ with $\pi(f(E'))=\overline{z}$, $\mult_{E'} B_Y\leq 1$;
 \item for any prime divisor $E$ on $Y$ with $\pi(f(E))\ni z$, $\mult_{E} B_Y\leq 1$;
 \item there exists an open neighborhood $U$ of $z\in Z$ such that $(\pi^{-1}(U), B|_{\pi^{-1}(U)})$ is lc.
\end{enumerate}

\end{lem}
\begin{proof}
By definition, (1) implies (2), (4) implies (3).
By direct computations, if (2) or (3) holds for the given log resolution $Y$, it holds for any log resolution. So 
(2) implies (1), and (3) implies (4).
It is obvious that (3) implies (2). It suffices to show that (1) implies (3).

Suppose that sub-pair $(X/Z\ni z, B)$ is lc. Assume to the contrary that exists a prime divisor $E$ such that $\mult_EB_Y>1$ and $E\cap f^{-1}\pi^{-1}(z)\neq \emptyset.$ Then by successively blowing up the closure of $E\cap f^{-1}\pi^{-1}(z)$ for several times, we can replace $Y$ by a higher model so that there exists a prime divisor $E'$ on $Y$ with $\pi(f(E'))=\overline{z}$ and $\mult_{E'} B_Y>1$ (cf. \cite[Corollary~2.31]{KM98}), a contradiction. 
\end{proof}

\begin{defn}
A \emph{non-klt place} of a sub-pair $(X, B)$ (respectively, $(X/Z\ni z, B)$) is a prime divisor $E$ over $X$ (respectively, $E\in \fD(X/Z\ni z)$) such that $a(E,X,B)\leq 0$, and a \emph{non-klt center} is the center of a non-klt place on $X$.
\end{defn}

\subsection{Log canonical thresholds}
\begin{defn}\label{definition of relative lct}
Let $(X/Z\ni z, B)$ be an lc sub-pair with contraction $\pi: X\to Z$, and let $D\neq 0$ be an effective $\mathbb{R}$-Cartier $\mathbb{R}$-divisor on $X$ such that $z\in \pi(\Supp(D))$. The \emph{log canonical threshold} of $D$ with respect to $(X/Z\ni z,B)$ is $$\lct(X/Z\ni z,B;D):=\sup\{t\in \mathbb{R}\mid (X/Z\ni z,B+tD)\text{ is lc}\}.$$

When $z\in Z$ is a codimension $1$ point, we may assume that $\overline{z}$ is a Cartier divisor on a neighborhood $U$ of $z\in Z$. Then we define $$\lct(X/Z\ni z,B;\pi^*\overline{z}):=\sup\{t\in \mathbb{R}\mid(X/Z\ni z,B+t\pi^*\overline{z})\text{ is lc over } U\},$$
and this definition does not depend on the choice of neighborhoods of $z\in Z$.
\end{defn}

We may write $\lct(X/Z\ni z; D):=\lct(X/Z\ni z, 0; D)$ when $B=0$. When $X=Z$, $z=x$, and $\pi$ is the identity map, 
we may write $\lct(X\ni x,B;D):=\lct(X/Z\ni z,B;D)$.

\begin{rem}Keep the same setting as in Definition~\ref{definition of relative lct}.
 Log canonical thresholds can be computed by a log resolution. In fact, take $g: X'\to X$ to be a log resolution of $(X, B+D)$ and write $K_{X'}+B'=g^*(K_X+B)$. Then 
 $$
 \lct(X/Z\ni z,B;D)=\min_E \frac{1-\mult_E(B')}{\mult_E g^*D}
 $$
 where the minimum runs over all prime divisors $E\subseteq\Supp g^*D$ such that $\pi(g(E))\ni z$ (cf. Lemma~\ref{lem lc near z}(3)). 
\end{rem}
 
\subsection{Canonical bundle formula}\label{sec cbf}

The {\it discrepancy b-divisor} $\mathbf{A}=\mathbf{A}(X,B)$ of a sub-pair $(X,B)$ is the b-divisor of $X$ with the trace $\mathbf{A}_Y$ defined by the formula
$$\mathbf{A}_Y=K_Y-f^*(K_X+B),$$
for any proper birational morphism $f: Y\to X$ between normal varieties.
Similarly, we define $\mathbf{A}^*=\mathbf{A}^*(X,B)$ by
$
\mathbf{A}^*_Y=\sum_{a_i>-1}a_iE_i
$ for any proper birational morphism $f: Y\to X$ between normal varieties,
where $\mathbf{A}_Y=\sum a_iE_i$. 
Note that $\mathbf{A}^*(X,B)=\mathbf{A}(X,B)$ if and only if $(X, B)$ is klt. See \cite[2.3]{FG14} for more details.

\begin{defn}[{\cite[Definition 3.2]{FG14}}]\label{def lctrivial}
 An {\it lc-trivial fibration} $\pi: (X, B)\to Z$ consists of a contraction $\pi: X\to Z$ between normal varieties and a sub-pair $(X, B)$ satisfying the following properties:
 \begin{enumerate}
 \item $(X, B)$ is lc over the generic point of $Z$;
 \item\label{lc-trivial condition 2} $\rank \pi_*\mathcal{O}_X(\lceil\mathbf{A}^*(X, B)\rceil)=1$;
 \item There exists an $\mathbb{R}$-Cartier $\mathbb{R}$-divisor $L$ on $Z$ such that $K_X+B\sim_{\mathbb{R}}\pi^*L.$
 \end{enumerate}
\end{defn}

\begin{rem}\label{rem lc B>0}
Here we discuss more details on condition~\eqref{lc-trivial condition 2}. 
If $B$ is effective on the generic fiber of $\pi$, then
$\mathcal{O}_X(\lceil\mathbf{A}^*(X, B)\rceil)=\mathcal{O}_X$ over the generic point of $Z$, so in this case condition~\eqref{lc-trivial condition 2} holds.
Conversely, if the generic fiber of $\pi$ is a rational curve, then $\rank \pi_*\mathcal{O}_X(\lceil\mathbf{A}^*(X, B)\rceil)=1$ implies that $B$ is effective on the generic fiber of $\pi$.
\end{rem}
 Let $\pi: (X, B)\to Z$ be an lc-trivial fibration. Then we may write $K_X+B\sim_{\Rr} \pi^*L$ for some $\mathbb{R}$-Cartier $\Rr$-divisor $L$. By the work of Kawamata \cite{Kaw97,Kaw98} and Ambro \cite{Amb05}, we have the so-called {\it canonical bundle formula}
$$K_X+B\sim_{\Rr}\pi^{*}(K_Z+B_Z+M_Z),$$
where $B_Z$ is defined by
\begin{align}\label{def BZ}
B_Z:=\sum_P (1-\lct(X/Z\ni \eta_P, B; \pi^*P))P
\end{align}
and 
\begin{align}\label{def MZ}M_Z:=L-K_Z-B_Z.
\end{align}
Here the sum runs over all prime divisors $P$ on $Z$ and $\eta_P$ is the generic point of $P$, and it is known that it is a finite sum. So $B_Z$ is uniquely determined by $(X, B)$ and $M_Z$ is determined up to $\mathbb{R}$-linear equivalences.
Here $B_Z$ is called the \emph{discriminant part} and $M_Z$ is called the \emph{moduli part} of the canonical bundle formula. Recall that if $B$ is effective, then $B_Z$ is also effective.

In the following, we suppose that $B$ is a $\mathbb{Q}$-divisor for simplicity. In fact, 
the {canonical bundle formula} satisfies certain functorial property as follows.
By \cite[Remark~7.7]{PS09} or \cite[3.4]{FG14}, there are b-divisors $\mathbf{B}$ and $\mathbf{M}$ of $Z$ such that
\begin{itemize}
 \item $\mathbf{B}_Z=B_Z$, $\mathbf{M}_Z=M_Z$, and
 \item for any birational contraction $g: Z'\to Z$, let $X'$ be a resolution of the main component of $X\times_{Z} Z'$ with induced morphisms $g': X' \to X$ and $\pi': X'\to Z'$. Let $K_{X'}+B'$ be the {\it crepant pull back} of $K_X+B$, that is, $K_{X'}+B'=g'^*(K_X+B)$, then $\mathbf{B}_{Z'}$ (respectively, $\mathbf{M}_{Z'}$) is the discriminant part (respectively, the moduli part) of the canonical bundle formula of $K_{X'}+B'$ on $Z'$
 defined by~\eqref{def BZ} and~\eqref{def MZ}.
 $$\xymatrix@R=2em{X' \ar[d]_{\pi'} \ar[r]^{g'} & X \ar[d]_{\pi}\\
 Z' \ar[r]_{g} & Z
 }
 $$
\end{itemize}

The effective adjunction conjecture 
(\cite[Conjecture~7.13]{PS09}) predicts that $\mathbf{M}$ is b-semi-ample. 
It was confirmed in the case of relative dimension $1$.

\begin{thm}[{\cite[Theorem~8.1]{PS09}}]
Keep the notation in this subsection.
If $\dim X-\dim Z=1$ and the generic fiber of $\pi$ is a rational curve, then $\mathbf{M}$ is b-semi-ample.
\end{thm}
\begin{rem}\label{rem PS8.1}
Note that \cite[Theorem~8.1]{PS09} holds for lc-trivial fibration $\pi: (X, B)\to Z$ under two additional assumptions: 
\begin{enumerate}[(i)]
\item $B$ is effective over the generic point of $Z$ \cite[Assumption~7.1]{PS09}, and 
\item there exists a $\mathbb{Q}$-divisor $\Theta$ on $X$ such that $K_X+\Theta\sim_{\mathbb{Q}, Z}0$ and $(X, \Theta)$ is klt over the generic point of $Z$
\cite[Assumption~7.11]{PS09}.
\end{enumerate}
Here (i) is automatically satisfied by Remark~\ref{rem lc B>0}. Also (ii) is automatically satisfied as the following.
Since the generic fiber $X_\eta$ of $\pi$ is a rational curve, we can find an effective $\mathbb{Q}$-divisor $D_\eta$ on 
$X_\eta$ such that $K_{X_\eta}+D_\eta\sim_{\mathbb{Q}}0$ and $(X_\eta, D_\eta)$ is klt.
Denote $D$ to be the closure of $D_\eta$ on $X$, then $K_{X}+D\sim_{\mathbb{Q}} E$ where $E$ is vertical over $Z$. Then we just take $\Theta=D-E$.
\end{rem}

\subsection{Contractions of Fano type}
\begin{defn}[{\cite{PS09}}]\label{defn RelFano}
Let $\pi: X\to Z$ be a contraction between normal varieties, we say that $X$ is \emph{of Fano type} over $Z$ if one of the following equivalent conditions holds:
\begin{enumerate}
 \item there exists a klt pair $(X,B)$ such that $-(K_X+B)$ is ample over $Z$;
 \item there exists a klt pair $(X,B')$ such that $-(K_X+B')$ is nef and big over $Z$;
 \item there exists a klt pair $(X,B'')$ such that $K_X+B''\equiv_Z 0$ and $B''$ is big over $Z$.
\end{enumerate}
When $Z$ is a point, we just say that $X$ is of Fano type.
\end{defn}

\subsection{Formal surface germs} 
Let $P$ be a smooth closed point on a surface $X$, by the Cohen structure theorem, $\widehat{\mathcal{O}}_{X,P}\cong \widehat{\mathcal{O}}_{\mathbb{C}^2,o}=\mathbb{C}[[x,y]]$. Denote by $\widehat {X}_P$ the completion of $X$ along $P$. We will use $\widehat {X}$ instead of $\widehat {X}_P$ if $P$ is clear from the context. 

We call ${C}$ a \emph{Cartier divisor} on $\widehat{X}$ if ${C}$ is defined by $(g=0)$ for some $g\in \widehat{\mathcal{O}}_{X,P}$. We call ${B}$ an \emph{$\mathbb{R}$-divisor} (respectively, a \emph{$\mathbb{Q}$-divisor}) on $\widehat{X}$ if ${B}=\sum_i b_i{B_i}$ for some Cartier divisors ${B_i}$ on $\widehat{X}$ and $b_i\in \mathbb{R}$ (respectively $b_i\in \mathbb{Q}$).

Since the resolution of singularities is known for complete local rings (\cite{Tem08}), the definition of singularities of pairs and log canonical thresholds can be extended to the formal case (see \cite{Kol08} and \cite{dFEM11}).

\begin{defn}\label{A formula calculating analytic lct}
Let $(\widehat{X}\ni P, {B}=\sum_i b_i{B_i})$ be a pair where $P\in X$ is a smooth formal surface germ and ${B_i}$ is defined by $(f_i=0)$ for some $f_i\in \widehat{\mathcal{O}}_{X,x}$. Let ${C}=\sum_i c_i{C_i}\neq 0$ be an effective $\mathbb{R}$-divisor, where ${C_i}$ is defined by $(g_i=0)$ for some $g_i\in \widehat{\mathcal{O}}_{X,x}$. Let ${\phi}: \widehat{Y}\to (\widehat{X}, {B}+{C})$ be a log resolution (\cite{Tem08}), then 
\begin{align}\label{def of formal lct}
 \lct(\widehat{X}\ni P, {B};{C}):=\min_{E}\frac{1+\mult_E K_{\widehat{Y}/\widehat{X}}-\sum_i b_i\mult_E (f_i)}{\sum_i c_i\mult_E (g_i)},
\end{align}
where the minimum runs over all prime divisors $E$ in $\Supp {\phi}^{*}C$ such that $P\in \phi(E)$.
The definition does not depend on the choice of log resolutions.
\end{defn}

\begin{rem}
Let $(X\ni P, B)$ be a germ of lc surface pair such that $P\in X$ is smooth, and let $C$ be an effective $\mathbb{R}$-divisor near $P$. Consider $\widehat{X}$ (respectively $B',C'$), the completion of $X$ (respectively $B,C$) along $P$. Since a log resolution of $(X\ni P, B+C)$ also gives a log resolution of $(\widehat{X},B'+C')$, $\lct(\widehat{X}\ni P, B';C')=\lct(X\ni P, B;C)$. In other words, in order to study the log canonical threshold of a smooth surface germ $(X\ni P,B)$, it is equivalent to study that of the corresponding smooth formal surface germ $(\widehat{X}\ni P, B')$. 
\end{rem}

Recall that log canonical thresholds satisfy convexity with respect to the coefficients.
\begin{lem}[cf. {\cite[Lemma~3.8]{HLQ17}}]\label{Convexity for lct} 
Let $P\in X$ be a smooth surface germ or a smooth formal surface germ.
Let $({X}\ni P,{B_i})$ be an lc pair for $1\leq i\leq m$, ${C}\neq 0$ an effective $\Rr$-divisor on ${X}$, $\lambda_i$ non-negative real numbers such that $\sum_{i=1}^m\lambda_i= 1$. Then 
$$\lct({X}\ni P,\sum_{i=1}^m \lambda_i {B_i};{C})\geq \sum_{i=1}^m \lambda_i \lct({X}\ni P,{B_i};{C}).$$ 
\end{lem}

\section{Log canonical thresholds on a smooth surface germ}\label{sec lct}

In this section, we study the lower bounds of log canonical thresholds on a smooth surface germ. The main goal of this section is to prove Theorem~\ref{Smooth local case general beta version}.

Recall the following result on computing log canonical thresholds of hypersurfaces.
\begin{prop}[{\cite[Proposition~2.1]{Kuw99}}]\label{computation lct fw}
Let ${B}$ be a Cartier divisor in a neighborhood of $o\in \widehat{\mathbb{C}^n}$ defined by $(f=0)$, where $f\in \mathbb{C}[[x_1, \dots, x_n]]$. Assign rational weights $w(x_i)$ to the variables and let $w(f)$ be the weighted multiplicity of $f$. Let $f_{w}$ denote the weighted homogeneous leading term of $f$. Take $b=\frac{\sum_{i=1}^n w(x_i)}{w(f)}$. If $(\widehat{\mathbb{C}^n}, b\cdot (f_w=0))$ is lc outside $o$, then $\lct(\widehat{\mathbb{C}^n}\ni o; B)=b$.
\end{prop}

To warm up, the following proposition is an application of Proposition~\ref{computation lct fw}.
\begin{prop}\label{prop x(x+y)}
Let ${B}$ be a Cartier divisor in a neighborhood of $o\in \widehat{\mathbb{C}^2}$ defined by $(f=0)$, where $f=x^{n } (x^{m_1}+y^{m_2})^k$ for some positive integers $k$, $n$, $m_1$, $m_2$. 
Then
$$\lct(\widehat{\mathbb{C}^2}\ni o; B)=\min\left\{
\frac{m_1+m_2}{km_1m_2+ {n }{m_2}}, \frac{1}{n }, \frac{1}{k}
\right\}.$$
 \end{prop}
\begin{proof}
Consider $C_1$ defined by $(x=0)$ and $C_2$ defined by $(x^{m_1}+y^{m_2}=0)$, then $(C_1\cdot C_2)_{o}=m_2.$ 
 Consider the weight $w=(m_2, m_1)$, then $f_w=f$ and $b=\frac{m_1+m_2}{km_1m_2+ {n }{m_2}}$ as in Proposition~\ref{computation lct fw}. 

If $b\leq \min\{ \frac{1}{n}, \frac{1}{k}\}$, then $(\widehat{\mathbb{C}^2}, b\cdot (f_w=0))$ is lc outside $o$, and hence $\lct(\widehat{\mathbb{C}^2}\ni o; B)=b$ by Proposition~\ref{computation lct fw}.
If $b> \frac{1}{n}$, then $n>km_2$. Then \cite[Corollary~5.57]{KM98} implies that $(\widehat{\mathbb{C}^2}\ni o, C_1+\frac{k}{n}C_2)$ is lc. 
If $b> \frac{1}{k}$, then either $m_1=1$ or $m_2=1$. In either case, $C_2$ is smooth and $k>nm_2$. Then \cite[Corollary~5.57]{KM98} implies that $(\widehat{\mathbb{C}^2}\ni o, \frac{n}{k}C_1+C_2)$ is lc. \end{proof}

\begin{defn}[cf. {\cite[Definition~2.10]{Kuw99}}]
Let $B=(f=0)$ be an irreducible curve in a neighborhood of $o\in \widehat{\mathbb{C}^2}$. If $B$ is smooth, then we set $m=1$ and $n=\infty$. Otherwise, the Puiseux expansion of $B$ (under suitable local parameters $x, y$) is expressed as $x=t^m, y=\sum_{i=n}^\infty \alpha_i t^i$ for some local parameter $t$, where $m,n\in\Zz_{\ge 2}$, $m<n$, and $m$ does not divide $n$. Here $(m, n)$ is called the \emph{first pair of Puiseux exponents of $f$}. Note that $m=\mult_o f$ is the multiplicity of $f$ at $o\in \widehat{\mathbb{C}^2}$. 
\end{defn} 

\begin{ex}
If $n>m>1$ and $m,n$ are coprime, then the first pair of Puiseux exponents of $f=x^m+y^n$ is just $(m,n)$. 
\end{ex}

The close relation between the first pair of Puiseux exponents and log canonical thresholds can be illustrated by the following result.
 \begin{thm}[{\cite[Theorem~1.3]{Kuw99}}]\label{dep on P}Let ${B}$ be a Cartier divisor in a neighborhood of $o\in \widehat{\mathbb{C}^2}$ defined by $(f=0)$, where $f\in \mathbb{C}[[x, y]]$. 
Write $f=\prod_{j=1}^r f_j^{\alpha_j}$ where $f_j$ is irreducible. Write $B=\sum_j \alpha_j B_j$ where $B_j$ is defined by $(f_j=0)$. Then 
$\lct(\widehat{\mathbb{C}^2}\ni o; B)$ depends only on the first pairs of Puiseux exponents of $f_j$, $(B_i\cdot B_j)_o$, and $\alpha_j$.
\end{thm}

Following the ideas in \cite[Theorem~1.2]{Kuw99}, we have the following.
\begin{prop}\label{Prop tB+sC lct}Let ${B}$ be a Cartier divisor in a neighborhood of $o\in \widehat{\mathbb{C}^2}$ defined by $(f=0)$, where $f\in \mathbb{C}[[x, y]]$. 
Suppose that $f$ is irreducible. Let $\mult_o f=m$ and let $(m, n)$ be the first pair of Puiseux exponents of $f$. Let $C\neq B$ be a smooth curve passing $o$ and $(B\cdot C)_o=I.$
Then for positive real number $s,t$, 
$$\lct(\widehat{\mathbb{C}^2}\ni o; sB+tC)=\min\left\{\frac{m+n}{smn+tI}, \frac{m+I}{(sm+t)I}, \frac{1}{s}, \frac{1}{t}\right\}.$$
\end{prop}
\begin{rem}\label{remark I=mp}
\begin{enumerate}
 \item By convention, if $(m,n)=(1,\infty)$, we set $\frac{1+\infty}{s\cdot \infty+tI}:=\frac{1}{s}$. 
 \item In the case that $s=t=1$, Proposition~\ref{Prop tB+sC lct} is a special case of \cite[Theorem~1.2]{Kuw99}. We also remark that Proposition~\ref{Prop tB+sC lct} might be indicated by more general results in \cite{GHM16}, but the formulation there is complicated and we give a simple proof in this special case for the reader's convenience.
 \item Recall that under the setting of Proposition~\ref{Prop tB+sC lct}, by \cite[Proof of Theorem~1.2, Case~2, Page 711--712]{Kuw99}, 
$$I\in \left\{m, 2m, \dots, \left\lfloor{\frac{n}{m}}\right\rfloor m, n\right\}.$$
\end{enumerate}
\end{rem}

\begin{proof}

Denote $$
c:=\min\left\{\frac{m+n}{smn+tI}, \frac{m+I}{(sm+t)I}, \frac{1}{s}, \frac{1}{t}\right\}.
$$
As being lc is a closed condition on coefficients, we may assume that $s, t\in \mathbb{Q}$.
Possibly replacing $s, t$ by a multiple, we may assume that $s, t$ are integers. 

If $m=1$, then by Theorem~\ref{dep on P}, we may assume that $sB+tC$ is defined by $(x^s (x +y^I)^t=0)$. Then the proposition follows from Proposition~\ref{prop x(x+y)}. In the following we may assume that $m>1$, and in particular, $B$ is singular at $o$. 

Suppose that $\frac{1}{s}\leq \frac{m+n}{smn+tI}$, then 
we have $m=1$ (recall that $n>1$), which is absurd. 
 
Suppose that $\frac{1}{t}\leq \frac{m+I}{(sm+t)I}$, then $sI\leq t$. Then \cite[Corollary~5.57]{KM98} implies that $(\widehat{\mathbb{C}^2}\ni o, \frac{s}{t}B+C)$ is lc. Since $n\ge I\ge m$, we have $\frac{m+n}{smn+tI}\ge \frac{1}{t}$, and hence $\frac{1}{t}=c$.

So from now on we may assume that
\begin{align}
 \frac{1}{s}> \frac{m+n}{smn+tI}\ \ \ \text{and}\ \ \ \frac{1}{t}> \frac{m+I}{(sm+t)I}, \label{1/s 1/t}
\end{align} in particular,
$$
c= \min\left\{\frac{m+n}{smn+tI}, \frac{m+I}{(sm+t)I}\right\}.
$$

If $I=n$, then by Theorem~\ref{dep on P}, we may assume that $sB+tC$ is defined by $((x^{m}+y^n)^s x^t=0)$. Then by Proposition~\ref{prop x(x+y)},
$$\lct(\widehat{\mathbb{C}^2}\ni o; sB+tC)=\min\left\{\frac{m+n}{smn+tn}, \frac{1}{s}, \frac{1}{t}\right\}=c.$$

If $I=pm$ for some $1\leq p\leq \lfloor{\frac{n}{m}}\rfloor$, then by Theorem~\ref{dep on P}, we may assume that $sB+tC$ is defined by $(h=0)$, where $h=(x^{m}+y^n)^s (x+y^p)^t$. 

If $tp\leq sm$, consider the weight $w=(n, m)$, then $h_w=y^{pt}(x^m+y^n)^s$ and $b=\frac{m+n}{smn+tI}$ as defined in Proposition~\ref{computation lct fw}.
Moreover, $(\widehat{\mathbb{C}^2},bh_w)$ is lc outside $o$ as $b\leq \frac{1}{pt}$ by $tp\leq sm$ and $b< \frac{1}{s}$ by~\eqref{1/s 1/t}. Hence by Proposition~\ref{computation lct fw},
$$\lct(\widehat{\mathbb{C}^2}\ni o; sB+tC)=\frac{m+n}{smn+tI}=c.$$

If $tp> sm$, consider the weight $w'=(p, 1)$, then $h_{w'}=x^{ms}(x+y^{p})^t$ and $b'=\frac{1+p}{(sm+t)p}=\frac{m+I}{(sm+t)I}$ as defined in Proposition~\ref{computation lct fw}.
Moreover, $(\widehat{\mathbb{C}^2},b'h_{w'})$ is lc outside $o$ as $b'< \frac{1}{ms}$ by $tp> sm$ and $b'<\frac{1}{t}$ by~\eqref{1/s 1/t}. 
Hence by Proposition~\ref{computation lct fw},
$$\lct(\widehat{\mathbb{C}^2}\ni o; sB+tC)=\frac{m+I}{(sm+t)I}=c.$$
\end{proof}

\begin{cor}\label{cor lct B C 1}
Let ${B}$ be a Cartier divisor in a neighborhood of $o\in \widehat{\mathbb{C}^2}$ defined by $(f=0)$, where $f\in \mathbb{C}[[x, y]]$. 
Suppose that $f$ is irreducible, $\mult_o f=m$ and let $(m, n)$ be the first pair of Puiseux exponents of $f$. Let $C\neq B$ be a smooth curve passing $o$, and $(B\cdot C)_o=I.$
Let $\lambda$ be a positive real number. Suppose that one of the following condition holds: 
\begin{enumerate}[(a)]
 \item $\lambda m \leq 1$; 
 \item $n=I$ and $\lambda\leq \min \{1, \frac{1}{m}+\frac{1}{I}\}$; or 
 \item $I\neq m$ and $\lambda I\leq 2$.
\end{enumerate}
Then $(\widehat{\mathbb{C}^2}\ni o, \lambda B)$ is lc and
$$\lct(\widehat{\mathbb{C}^2}\ni o, \lambda B; C)\geq \min\left\{1, 1+\frac{m}{I}-\lambda m\right\}.$$
\end{cor}

\begin{proof}
Here note that under condition (a) or (c), $\lambda\leq \min \{1, \frac{1}{m}+\frac{1}{I}\}$ automatically holds.
Denote $t:=\min\{1, 1+\frac{m}{I}-\lambda m\}\geq 0.$ 
The statement is equivalent to
$\lct(\widehat{\mathbb{C}^2}\ni o, \lambda B+tC)\geq 1$.
By Proposition~\ref{Prop tB+sC lct}, this is equivalent to show that 
 \begin{enumerate}
 \item\label{3.8 proof condition 1} $\frac{m+n}{\lambda mn+tI}\geq 1$,
 \item\label{3.8 proof condition 2} ${m+I}\geq {(\lambda m+t)I}$, 
 \item\label{3.8 proof condition 3} $1\geq \lambda$, and
 \item\label{3.8 proof condition 4} $1\geq t$.
 \end{enumerate}
 Here~\eqref{3.8 proof condition 2} and~\eqref{3.8 proof condition 4} follow from the definition of $t$, and~\eqref{3.8 proof condition 3} follows from the condition on $\lambda$. 
 To show~\eqref{3.8 proof condition 1}, we may assume that $m\ge 2$. It suffices to prove that 
 $${m+n}\geq {\lambda mn+\left(1+\frac{m}{I}-\lambda m\right)I},$$
 which is equivalent to $(n-I)(1-\lambda m)\geq 0$.
 Recall that $n\geq I$, so~\eqref{3.8 proof condition 1} holds if either $n=I$ or $\lambda m\leq 1$ holds. This proves the conclusion for (a) and (b).
 To conclude the proof, we want to show that if (c) holds, then either (a) or (b) holds. In fact, suppose that $\lambda I\leq 2$ and $\lambda m > 1$, then $I<2m$. Then by Remark~\ref{remark I=mp}(3), $I=n$.
 \end{proof}

 \begin{rem}
In applications, we only use Corollary~\ref{cor lct B C 1} when condition (a) holds. The advantage of this corollary is that we can get rid of $n$ in the first pair of Puiseux exponents of $f$ and the log canonical threshold can be estimated by only $m$ and $I$. In practice, $n$ is usually hard to control, while $m$ and $I$ can be controlled easily by geometric conditions. 
 \end{rem}
 
The following example shows that both Theorem~\ref{Smooth local case general beta version} and Corollary~\ref{cor lct B C 1} are optimal.
\begin{ex}\label{ex optimal lower bound}
Given two coprime positive integers $m$ and $I$ such that $m
< I$. Take a positive real number $\lambda$ such that $\lambda m\leq 1 \leq \lambda I$.
Consider $(\mathbb{C}^2, \lambda B)$ where $B=(x^{m}+y^{I}=0)$ and $C=(x=0)$. Then $\mult_o \lambda B=\lambda m$, $(\lambda B\cdot C)_o= \lambda I$.
A direct computation by Proposition~\ref{Prop tB+sC lct} shows that 
$(\mathbb{C}^2\ni o, \lambda B+(1+\frac{m}{I}-\lambda m)C)$ is lc but $(\mathbb{C}^2\ni o, \lambda B+(1+\frac{m}{I}-\lambda m+\epsilon )C)$ is not lc for any $\epsilon>0$. So in this case 
$$\lct(\mathbb{C}^2\ni o, \lambda B;C)= 1+\frac{m}{I}-\lambda m.$$
\end{ex}

Now we may show Theorem~\ref{Smooth local case general beta version} which could be regarded as an $\mathbb{R}$-divisor version of Corollary~\ref{cor lct B C 1}.

\begin{proof}[Proof of Theorem~\ref{Smooth local case general beta version}]

If $I\leq 1$, then $(X\ni P,B+C)$ is lc by \cite[Corollary~5.57]{KM98}. Hence we may assume that $I>1$.

We may replace $P\in X$ by the formal neighborhood $\widehat{X}$ of $P\in X$, which is isomorphic to the formal neighborhood $o\in \widehat{\mathbb{C}^2}$. So from now on we may assume that $P\in X$ is just $o\in \widehat{\mathbb{C}^2}$. Write ${B}=\sum_{i=1}^{n} b_i {B_i}$, where $b_i\in(0,1]$, and $\{B_i\}_{1\leq i\leq n}$ are distinct irreducible curves on $\widehat{\mathbb{C}^2}$ passing $o$.

If $n=1$, then we are done by Corollary~\ref{cor lct B C 1}. So we may assume that $n\geq 2$.

Set $s:=1+\frac{m}{I}-m$. The goal is to show that $(\widehat{\mathbb{C}^2} \ni o,B+s{C})$ is lc. Consider the {\it log canonical threshold polytope} of the pair $(\widehat{\mathbb{C}^2} \ni o,s{C})$ with respect to the divisors ${B_1},\ldots,{B_n}$, $$P(\widehat{\mathbb{C}^2}\ni o,s{C};{B_1},\ldots,{B_n}):=\left\{(t_1,\ldots,t_n)\in \mathbb{R}^n_{\geq 0}\,\bigg\vert \left(\widehat{\mathbb{C}^2}\ni o ,s{C}+\sum_{i=1}^n t_i{B_i}\right)\text{ is lc}\right\}.$$
By Lemma~\ref{Convexity for lct}, $P(\widehat{\mathbb{C}^2}\ni o,s{C};{B_1},\ldots,{B_n})$ is a compact convex polytope in $\mathbb{R}^n$. It suffices to show that the convex polytope
$$\mathcal{P}:=\left\{(t_1,\ldots,t_n)\in\mathbb{R}^n_{\geq 0}\,\bigg\vert \mult_o\sum_{i=1}^n t_i{B_i}=m,\ \sum_{i=1}^n t_i({B_i}\cdot{C})_o=I\right\}$$ is contained in $P(\widehat{\mathbb{C}^2}\ni o ,s{C};{B_1},\ldots,{B_n})$. By Lemma~\ref{Linear algebra}, all the vertices of $\mathcal{P}$ are contained in $\bigcup_{i\neq j} E_{i,j}$, where $E_{i,j}:=\{(t_1,\cdots,t_n)\mid t_k=0 \text{ for } k\neq i,j\}$. Hence it suffices to show that $$E_{i,j}\cap\mathcal{P} \subseteq E_{i,j}\cap P(\widehat{\mathbb{C}^2} \ni o,s{C};{B_1},\ldots,{B_n})\simeq P(\widehat{\mathbb{C}^2} \ni o,s{C};{B_i},{B_j})$$ for all $1\leq i< j\leq n$.

Without loss of generality, we may just consider the case $(i, j)=(1,2)$.
It suffices to show that any vertex point of $E_{1,2}\cap\mathcal{P}$ is contained in $P(\widehat{\mathbb{C}^2} \ni o,s{C};{B_1},{B_2})$, where $E_{1,2}$ is identified with $\mathbb{R}^2$. 
Denote $\mult_o B_i=m_i, (B_i\cdot C)_o=I_i\ge 1$ for $i=1,2$.
 Take $(c_1, c_2)$ to be a vertex point of $E_{1,2}\cap\mathcal{P}$, then $(c_1, c_2)$ 
satisfies the following equations
\begin{align}\label{equations for b}
 m_1c_1+m_2c_2=m,\ \ \ I_1c_1+I_2c_2=I.
\end{align}
Here we recall that $m_1, m_2, I_1, I_2$ are positive integers, $m_1\leq I_1$, $m_2\leq I_2$, and $m\leq 1<I$.

Suppose that either $c_1=0$ or $c_2=0$, then $(c_1, c_2)\in P(\widehat{\mathbb{C}^2}\ni o,s{C};{B_1},{B_2})$ follows directly from Corollary~\ref{cor lct B C 1}.

Suppose that $c_1>0$ and $c_2>0$. Since $(c_1, c_2)$ is a vertex of $E_{1,2}\cap\mathcal{P}$, it is the unique solution of~\eqref{equations for b}. Thus $\frac{m_1}{I_1}\neq \frac{m_2}{I_2}$, and 
$$\min\{\frac{m_1}{I_1},\frac{m_2}{I_2}\}< \frac{m_1c_1+m_2c_2}{I_1c_1+I_2c_2}=\frac{m}{I}< \max\{\frac{m_1}{I_1},\frac{m_2}{I_2}\}.$$
Without loss of generality, we may assume that $\frac{m_1}{I_1}<\frac{m}{I}<\frac{m_2}{I_2}$. See Figure~\ref{2 case for dim 2 SM conj}.

\begin{figure}[ht]

\begin{tikzpicture}[scale=0.8] 
\draw[thick,->] (0,0) -- (5,0) node[anchor=north west] {};
\draw[thick,->] (0,0) -- (0,5) node[anchor=south east] {};
\draw[thick,-] (0,4.1)--(1.7,3.7);
\draw[thick,-] (1.7,3.7)--(2.8,2.8);
\draw[thick,-] (2.8,2.8)--(3.7,1.7);
\draw[thick,-] (3.7,1.7)--(4.1,0);
\draw[thick,-] (0,3) -- (1.8,1.8);
\draw[thick,-] (1.8,1.8)--(3,0);
\draw[dashed,thick,-] (1.8,1.8)--(4.5,0);
\node[below] at (4.5,0) {\tiny$(\frac{m}{m_1}, 0)$};
\node[left] at (0,3) {\tiny$(0, \frac{m}{m_2})$};
\node[below] at (3,0) {\tiny$( \frac{I}{I_1},0)$};
\node[below left] at (0,0) {\tiny$\mathbf{o}$};
\node[above right] at (1.7,1.7) {\tiny$(c_1,c_2)$};
\node[] at (3,4.4) {\tiny$P(\widehat{\mathbb{C}^2},s{C};{B_1},{B_2})$};
\node[] at (1.2,1.2) {};
\node[] at (2.5,-1) {\small When $m\geq\frac{m_2}{I_2}>\frac{m_1}{I_1}$.};
\end{tikzpicture}
\begin{tikzpicture}[scale=0.8] 
\draw[thick,->] (0,0) -- (5,0) node[anchor=north west] {};
\draw[thick,->] (0,0) -- (0,5) node[anchor=south east] {};
\draw[thick,-] (0,4.1)--(1.7,3.7);
\draw[thick,-] (1.7,3.7)--(2.8,2.8);
\draw[thick,-] (2.8,2.8)--(3.7,1.7);
\draw[thick,-] (3.7,1.7)--(4.1,0);
\draw[thick,-] (0,3) -- (1.8,1.8);
\draw[thick,-] (1.8,1.8)--(3,0);
\draw[dashed,thick,-] (4.3,0)--(1.8,1.8);
\draw[dashed,thick,-] (0,3.3)--(1.8,1.8);
\node[left] at (0,3.5) {\tiny$(0, \frac{1}{I_2})$};
\node[left] at (0,3) {\tiny$(0, \frac{m}{m_2})$};
\node[below] at (3,0) {\tiny$( \frac{I}{I_1},0)$};
\node[below] at (4.38,0) {\tiny$(\lambda_1, 0)$};
\node[below left] at (0,0) {\tiny$\mathbf{o}$};
\node[above right] at (1.7,1.7) {\tiny$(c_1,c_2)$};
\node[] at (3,4.4) {\tiny$P(\widehat{\mathbb{C}^2},s{C};{B_1},{B_2})$};
\node[] at (1.2,1.2) {};
\node[] at (2.5,-1) {\small When $\frac{m_2}{I_2}>m$.};
\end{tikzpicture}
\caption{}\label{2 case for dim 2 SM conj}
\end{figure}

If $m\geq\frac{m_2}{I_2}>\frac{m_1}{I_1}$, then we may write $c_1B_1+c_2B_2=\mu_1 \frac{m}{m_1}B_1+\mu_2 \frac{m}{m_2}B_2$ for $\mu_1=\frac{m_1c_1}{m}$ and $ \mu_2=\frac{m_2c_2}{m}$. Note that $\mu_1+\mu_2=1$. By Corollary~\ref{cor lct B C 1} and $m\leq 1$, $$\lct(\widehat{\mathbb{C}^2}\ni o,\frac{m}{m_i}B_i;C)\geq \min\left\{1, 1+\frac{m_i}{I_i}-m\right\}=1+\frac{m_i}{I_i}-m$$ for $i=1,2$. By Lemma~\ref{Convexity for lct} and the Cauchy--Schwarz inequality, we have
\begin{align*}
 \lct(\widehat{\mathbb{C}^2}\ni o,c_1B_1+c_2B_2;C)&\geq \mu_1 \lct(\widehat{\mathbb{C}^2}\ni o,\frac{m}{m_1}B_1;C)+\mu_2 \lct(\widehat{\mathbb{C}^2}\ni o,\frac{m}{m_2}B_2;C)\\
 &\geq 1-m+\mu_1\frac{m_1}{I_1}+\mu_2\frac{m_2}{I_2}=1-m+\frac{m_1^2c_1}{I_1m}+\frac{m_2^2c_2}{I_2m}\\
 &\geq 1-m+ \frac{(m_1c_1+m_2c_2)^2}{(I_1c_1+I_2c_2)m}=1-m+\frac{m}{I}=s.
\end{align*}

Otherwise, $\frac{m_2}{I_2}>m$. We may write 
 $c_1B_1+c_2B_2=\mu'_1 \lambda_1B_1+\mu'_2 \frac{1}{I_2}B_2$, where $\mu'_2=I_2c_2$, $\mu'_1=1-I_2c_2$, $\lambda_1=\frac{c_1}{1-I_2c_2}$. Note that $\mu'_1> 1-\frac{m_2c_2}{m}> 0$, $\mu'_1+\mu'_2=1$, and $\lambda_1\leq \frac{c_1}{1-\frac{m_2c_2}{m}}=\frac{m}{m_1}\leq \frac{1}{m_1}$.
By Corollary~\ref{cor lct B C 1}, we have
$$\lct(\widehat{\mathbb{C}^2}\ni o,\lambda_1B_1;C)\geq \min\left\{1, 1+\frac{m_1}{I_1}-\lambda_1 m_1\right\}\ \ \ \text{and}\ \ \ \lct(\widehat{\mathbb{C}^2}\ni o,\frac{1}{I_2}B_2;C)\geq 1.$$ 
By Lemma~\ref{Convexity for lct}, we have
\begin{align*}
 \lct(\widehat{\mathbb{C}^2}\ni o,c_1B_1+c_2B_2;C)
 &\geq \mu'_1 \lct(\widehat{\mathbb{C}^2}\ni o, \lambda_1 B_1;C)+\mu'_2 \lct(\widehat{\mathbb{C}^2}\ni o,\frac{1}{I_2}B_2;C)\\
 &\geq \min\left\{1, 1+\mu'_1m_1(\frac{1}{I_1}-\lambda_1)\right\}= \min\left\{1, 1+\frac{m_1}{I_1}(1-I)\right\}\\
 &\geq\min\left\{1, 1+\frac{m}{I}(1-I)\right\}=s.
\end{align*}
Here for the equality we use the fact that
$$
\mu'_1(\frac{1}{I_1}-\lambda_1)=\frac{1-I_2c_2-I_1c_1}{I_1}=\frac{1-I}{I_1}.
$$

In summary, we have showed that $(c_1, c_2)\in P(\widehat{\mathbb{C}^2}\ni o,s{C};{B_1},{B_2})$, and the proof is completed.
\end{proof}

\begin{lem}\label{Linear algebra}
Let $b_j\geq 0$ and $\mathbf{n}_j\in \Rr^n_{>0}$ for $j=1,2$. Assume that $n\geq 2$, then $$\mathcal{P}:=\{\mathbf{t}\in\Rr^n_{\ge0}\mid \langle \mathbf{n_j},\mathbf{t}\rangle= b_j,\,j=1,2\}$$ is a convex polytope, and all the vertices of $\mathcal{P}$ belong to $\bigcup_{1\leq i\neq j\leq n} E_{i,j}$, where $$E_{i,j}:=\{(t_1,\cdots, t_n)\in \mathbb{R}^n\mid \text{$t_k=0$ for $k\neq i,j$}\}.$$
\end{lem}

\begin{proof}
It is easy to check that $\mathcal{P}$ is a convex polytope of dimension at least $n-2$.
Note that each vertex of $\mathcal{P}$ belongs to at least $n-2$ faces of $\mathcal{P}$. Since $\mathcal{P}$ has at most $n$ faces $\{(t_1,\cdots,t_n)\in\Rr^n\mid t_i=0\}\cap \mathcal{P}$ for $i=1,2,\ldots,n$, we conclude that each vertex of $\mathcal{P}$ belongs to $\bigcup_{1\leq i< j\leq n} E_{i,j}$.
\end{proof}

\begin{cor}\label{Smooth local case}
Let $(X\ni P, B)$ be a germ of surface pair such that $X$ is smooth and $\mld(X\ni P, B)\geq 1$. Let $C$ be an smooth curve at $P$ such that $C\nsubseteq\Supp B$ and $(B\cdot C)_P\le 2$. Then $\lct(X\ni P,B;C)\geq \frac{1}{2}$.
\end{cor}
\begin{proof}
Note that $\mld(X\ni P, B)\geq 1$ implies that $m:=\mult_{P} B\le 1$ (cf. \cite[Lemma 3.15]{HL20}). By Theorem~\ref{Smooth local case general beta version} for the case when $I\leq 2$,
$$\lct(X\ni P,B;C)\geq \min\{1,1+\frac{m}{I}-m\}\geq 1+\frac{m}{2}-m\ge \frac{1}{2}.$$
\end{proof}

\section{Proofs of the main theorems}\label{sec proofs}
\subsection{Proof of Theorem~\ref{High dim}}

In this subsection, we give the proof of Theorem~\ref{High dim}.
We first treat the case when $\dim X=2$.

\begin{proof}[Proof of Theorem~\ref{High dim} when $\dim X=2$]
We split the proof into two steps.

\medskip

\noindent {\bf Step 1}. First we treat the case when $X$ is smooth, $B\geq 0$, and $\mld(X/Z\ni z, B)= 1$.

\medskip

As the generic fiber of $\pi$ is a rational curve,
we may run a $K_X$-MMP over $Z$ and reach a minimal ruled surface $\pi':X'\to Z$. Denote by $\phi: X\to X'$ the induced morphism and $B'=\phi_*B$. Since $K_{X}+B\sim_{\mathbb{R},Z} 0$, by the negativity lemma \cite[Lemma~3.39]{KM98}, $\phi^*(K_{X'}+B')=K_X+B$. Thus $K_{X'}+B'\sim_{\mathbb{R},Z} 0$, $\mld(X'/Z\ni z, B')=\mld(X/Z\ni z, B)$, and $\lct(X'/Z\ni z, B';\pi'^*z)=\lct(X/Z\ni z, B;\pi^*z)$. 
Now $F:=\pi'^{*}(z)\cong \mathbb{P}^1$, and $(K_{X'}+B')\cdot F=0$. By the adjunction formula, $K_{X'}\cdot F=-2$. Hence $(B'\cdot F)_P\leq 2$ for any closed point $P\in F$. Recall that $\mld(X'/Z\ni z, B')= 1$ implies that $F\nsubseteq \Supp B'$.
By Corollary~\ref{Smooth local case}, $\lct(X'\ni P, B';F)\geq \frac{1}{2}$ for any closed point $P\in F$, which implies that $\lct(X'/Z\ni z, B'; \pi'^{*}z)\geq \frac{1}{2}$. Hence $\lct(X/Z\ni z, B; \pi^{*}z)\geq \frac{1}{2}$.

\medskip

\noindent {\bf Step 2}. We treat the general case.

\medskip

Write $\mld(X/Z\ni z, B)= 1+\epsilon$ for some $\epsilon\geq 0$.
Let $f:W\to X$ be a log resolution of $(X,B+\pi^*z)$. 
We may write $K_W+B_W=f^{*}(K_X+B).$
Since $\mld(X/Z\ni z, B)= 1+\epsilon$, for any curve $C\subset \Supp f^*\pi^*z$, $\mult_C B_W\leq -\epsilon$. We can take $s\geq 0$ such that for any curve $C\subset \Supp f^*\pi^*z$, $\mult_C (B_W+sf^*\pi^*z)\leq 0$, and there exists a curve $C_0\subset \Supp f^*\pi^*z$ with $\mult_{C_0} (B_W+sf^*\pi^*z)= 0$. By Lemma~\ref{lem lc near z}, possibly shrinking $Z$ near $z$, we may assume that $(X, B)$ is lc, so the coefficients of $B_W$ are at most $1$. Since $B_W+sf^*\pi^*z$ is a simple normal crossing divisor, by \cite[Lemma 3.3]{CH21}, $\mld(W/Z\ni z, B_W+sf^*\pi^*z)= 1$. Note that $B_W+sf^*\pi^*z$ is not necessarily effective, so we can not apply Step 1 directly.

We may write $B_W+sf^*\pi^*z=D-G$, where $D$ and $G$ are effective 
$\mathbb{R}$-divisors with no common components.
Then 
$$K_W+D=f^*(K_X+B+s\pi^*z)+G\sim_{\mathbb{R}, Z} G.
$$ 
By Remark~\ref{rem lc B>0}, $B$ is effective on the generic fiber of $\pi$, so $\Supp G$ does not dominate $Z$. Possibly shrinking $Z$ near $z$, we may assume that $\Supp G\subset \Supp f^{*}\pi^*z$.
By the construction, $C_0\subset \Supp f^{*}\pi^*z$ but $C_0\not \subset \Supp G$. Note that $(W, D)$ is lc as the coefficients of $D$ are at most $1$.

If $E$ is a curve on $W$ with $(K_W+D)\cdot E<0$, then $G\cdot E<0$ and hence $E\subset \Supp G$. Then $E\not\subset \Supp D$, and $K_W\cdot E<0$. This implies that any $(K_W+D)$-MMP over $Z$ is also a $K_W$-MMP over $Z$, and it only contracts curves in $\Supp G$.

We may run a $(K_W+D)$-MMP over $Z$ and reach a minimal model $Y$ with induced maps $g:W\to Y$ and $h:Y \to Z$, such that $K_Y+D_Y\sim_{\Rr,Z}G_Y$ is nef over $Z$, where $D_Y$ and $G_Y$ are the strict transforms of $D$ and $G$ on $Y$ respectively.

As this MMP is also a $K_W$-MMP, $Y$ is a smooth surface. 
Recall that $C_0\not\subseteq \Supp G$, so $C_0$ is not contracted by this MMP and $\Supp G_Y \subsetneq \Supp h^{*}z$. Hence $G_Y=0$ as $G_Y$ is nef over $Z$. 
Since $K_Y+D_Y=g_{*}(K_W+D-G)\sim_{\Rr,Z}0$, by the negativity lemma \cite[Lemma~3.39]{KM98}, $$g^{*}(K_Y+D_Y)=K_W+B_W+sf^*\pi^*z=f^{*}(K_X+B+s\pi^*z)\sim_{\Rr,Z}0.$$ Thus $\mld(Y/Z\ni z, D_Y)=\mld(W/Z\ni z, B_W+sf^*\pi^*z)=1$, and $$\lct(Y/Z\ni z, D_Y;h^{*} z)=\lct(X/Z\ni z, B+s\pi^*z;\pi^{*}z)= \lct(X/Z\ni z, B;\pi^{*}z)-s.$$ Since $X$ and $Y$ are isomorphic over the generic point of $Z$, the generic fiber of $h$ is again a rational curve. So $(Y,D_Y)$ satisfies the setting in Step 1.
By Step 1, we get $\lct(Y/Z\ni z, D_Y; h^{*}z)\geq \frac{1}{2}$.

To conclude the proof, we need to give a lower bound for $s$. As $Y$ is smooth, $Y$ dominates a $\mathbb{P}^1$-bundle over $Z$. So there exists a curve 
$C_1$ on $Y$ such that
$C_1\subset \Supp h^{*}z$ and $\mult_{C_1} h^{*}z=1$.
Denote $C_1'$ to be the strict transform of $C_1$ on $W$, then
$C'_1\subset \Supp f^*\pi^*z$ and $\mult_{C'_1} f^*\pi^*z=1$.
Note that $\mult_{C'_1} (B_W+sf^*\pi^*z)=\mult_{C_1} (D_Y)\geq 0$. On the other hand, $\mult_{C'_1} (B_W+sf^*\pi^*z)\leq 0$ by the definition of $s$. So $\mult_{C'_1} (B_W+sf^*\pi^*z)= 0$. As $\mult_{C'_1} B_W\leq -\epsilon$, we have $s\geq \epsilon$.
 Hence 
 $$ \lct(X/Z\ni z, B;\pi^{*}z)=\lct(Y/Z\ni z, D_Y;h^{*} z)+s\geq \frac{1}{2}+\epsilon=\mld(X/Z\ni z, B)-\frac{1}{2}.$$ 
 This concludes the proof.
\end{proof}

Next we give the proof of Theorem~\ref{High dim} by induction on dimensions.

\begin{proof}[Proof of Theorem~\ref{High dim}]

We prove the theorem by induction on the dimension of $X$. 
We have proved the case when $\dim X=2$. Suppose that Theorem~\ref{High dim} holds when $\dim X=n$ for some integer $n\ge 2$, we will show that the theorem holds when $\dim X=n+1$. 
 
As the statement is local around $z\in Z$, we are free to shrink $Z$. Possibly shrinking $Z$ near $z$, we may assume that $\overline{z}$ is a Cartier divisor on $Z$. Denote $t:=\lct(X/Z\ni z, B; \pi^*{\overline{z}})$. Possibly shrinking $Z$ near $z$, we may assume that $(X, B+ t\pi^*\overline{z})$ is lc.

Pick a general hyperplane section $H\subset Z$ intersecting $\overline{z}$. Possibly shrinking $Z$ near $z$, we may assume that $H\cap \overline{z}$ is irreducible. Let $z_H$ be the generic point of $H\cap\overline{z}$ and $G:=\pi^* H$, then by the Bertini's theorem, the restriction $\pi_G=\pi|_G: G\to H$ is a contraction between normal varieties such that $K_G+B|_G\sim_{\mathbb{R},H} 0$. 
Since $H$ is general, 
we may assume that 
\begin{itemize}

 \item the generic fiber of $\pi_G$ is a rational curve, and 
 \item $(X, B+G+t\pi^*\overline{z})$ is lc.
\end{itemize}
Let $\phi: Y\to X$ be a log resolution of $(X,B+\pi^*\overline{z})$, we may write $$K_Y+\phi^{-1}_*B+\sum_i (1-a_i)E_i=\phi^*(K_X+B),$$ where $E_i$ are $\phi$-exceptional prime divisors. Possibly shrinking $Z$ near $z$, we may further assume that $z\in \pi\circ \phi (E_i)$ for each $i$. By taking $H$ general enough, we may assume that
\begin{itemize}

\item $\phi^*G=\phi_{*}^{-1}G$,
 and
 \item $\phi$ is a log resolution of $(X,B+\pi^*\overline{z}+G)$.
\end{itemize}
Note that as $\phi_{*}^{-1}G=\phi^*G=\phi^{*}\pi^*H$, we have $\pi\circ\phi(E_i\cap \phi_{*}^{-1}G)=\pi\circ\phi(E_i)\cap H$ for each $i$.

Since $$K_Y+\phi^{-1}_*B+\phi^{-1}_{*}G+\sum_i (1-a_i)E_i=\phi^*(K_X+B+G),$$ by the adjunction formula \cite[Proposition~5.73]{KM98}, $$K_{\phi_{*}^{-1}G}+\phi^{-1}_*B|_{\phi_{*}^{-1}G}+\sum_i (1-a_i)E_i|_ {\phi_{*}^{-1}G}=\phi^*(K_G+B|_G),$$ 
which implies that the induced morphism $\phi_{*}^{-1} (G)\to G$ is a log resolution of $(G, B|_G+ \pi_G^*\overline{z_H})$. 
Note that $z$ and $z_H$ are codimension $1$ points of $Z$ and $H$ respectively, we have 
\begin{align*}
 \mld(G/H\ni z_H, B|_G){}&=\min\{a_i\mid \pi\circ\phi(E_i\cap \phi_{*}^{-1}G)=\overline{z_H} \}\\{}&= \min\{a_i\mid \pi\circ\phi(E_i)=\overline{z} \}=\mld(X/Z\ni z,B).
\end{align*}
Similarly, we have 
 $$K_Y+\phi^{-1}_*B+\phi^{-1}_{*}G+t\phi^{-1}_{*}\pi^*\overline{z}+\sum_i (1-a'_i)E_i=\phi^*(K_X+B+G+t\pi^*\overline{z}),$$ $$K_{\phi_{*}^{-1}G}+\phi^{-1}_*B|_{\phi_{*}^{-1}G}+t\phi^{-1}_{*}\pi^*\overline{z}|_{\phi_{*}^{-1}G}+\sum_i (1-a_i')E_i|_ {\phi_{*}^{-1}G}=\phi^*(K_G+B|_G+t \pi_G^*\overline{z_H}).$$ 
 As $(X, B+G+t\pi^*\overline{z})$ is lc, so is 
 $(G, B|_G+t \pi_G^*\overline{z_H})$. On the other hand, by the definition of $t$, there exists an index $i$ such that $a_i'=0$ and $E_i \subseteq \Supp(\phi^*\pi^*\overline{z})$. In particular, $ \pi\circ \phi (E_i)=\overline{z}$.
Then by the construction, $E_{i}\cap \phi^{-1}_{*}G\neq \emptyset$, 
which gives a non-klt place of $(G, B|_G+t \pi_G^*\overline{z_H})$ whose image on $H$ is $\overline{z_H}$.
 Thus $t=\lct(G/H\ni z_{H},B|_G; \pi_G^{*}\overline{z_H})$.
 As $(G/H\ni z_{H},B|_G)$ satisfies the condtions of Theorem~\ref{High dim},
 \begin{align*}
 &\lct(X/Z\ni z, B; \pi^*{\overline{z}})={}\lct(G/H\ni z_{H},B|_G; \pi_G^{*}\overline{z_H})\\
 \geq{}& \mld(G/H\ni z_H, B|_G)-\frac{1}{2}
 ={}\mld(X/Z\ni z,B)-\frac{1}{2}
 \end{align*}
 by the induction hypothesis.

For the last statement, note that $\lct(X/Z\ni z, B; \pi^*\overline{z})\geq \frac{1}{2}$ implies that the coefficients of $B+\frac{1}{2}\pi^*\overline{z}$ are at most $1$ over a neighborhood of $z\in Z$.
So if $B$ is effective, then the multiplicity of each irreducible component of $\pi^*{z}$ is
bounded from above by $2$.
\end{proof}

The following example shows that the
bounds in Theorems~\ref{High dim SM} and~\ref{High dim} are optimal.

\begin{ex}\label{ex:lct is attained}
Consider $C\simeq \mathbb{P}^1$. Consider $Y=C\times \mathbb{P}^1$ and the natural projection $\pi: Y\to C$. Take $D$ to be a smooth curve on $Y$ of type $(1,2)$. 
Note that there exists a closed point $p\in C$ such that $D$ intersects $\pi^{-1} (p)$ at a single closed point with intersection multiplicity $2$.
Denote $F=\pi^{-1} (p)$. Then for any real number $s\geq 0$, we consider the sub-pair $(Y, D-sF)$. We can get a log resolution of $(Y, D-sF)$ by blowing up twice as the following. 
Let $Y_1\to Y$ be the blow-up at $F\cap D$. Denote by $F_1, D_1$ the strict transforms of $F, D$ on $Y_1$ respectively, and $E_1$ the exceptional divisor.
Then $F_1, D_1, E_1$ intersect at one point.
Let $Y_2\to Y_1$ be the blow-up at $F_1\cap D_1\cap E_1$, denote by $F_2,D_2, E_2$ the strict transforms of $F_1,D_1, E_1$ on $Y_2$ respectively, and $G_2$ the exceptional divisor on $Y_2$.
Then $Y_2$ is a log resolution of $(Y, D-sF)$.
Denote $\pi: Y_2\to C$ and $f: Y_2\to Y$ the induced maps. Then we have
$$
K_{Y_2}+D_2-sF_2-sE_2-2sG_2=f^*(K_Y+D-sF)\sim_{\mathbb{R}, C} 0.
$$
and
$$
\pi^*p=f^*F=F_2+E_2+2G_2.
$$
Denote $B_2=D_2-sF_2-sE_2-2sG_2$.
Then $(Y_2/C\ni p, B_2)$ satisfies the conditions of Theorem~\ref{High dim}.
It is easy to compute that $\mld(Y_2/C\ni p, B_2)=1+s$ and $\lct(Y_2/C\ni p, B_2;\pi^{*}p)=\frac{1}{2}+s$.
Also we have $\mult_{G_2}\pi^*p=2.$
This shows that Theorem~\ref{High dim} is optimal.

In this case, if we consider the canonical bundle formula of $(Y_2, B_2)$ over $C$, then the discriminant part $B_C=(\frac{1}{2}-s)p$, and hence for any $M_C\geq 0$ on $C$,
$$
\mld(C\ni p, B_C+M_C)\leq \mld(C\ni p, B_C)=\frac{1}{2}+s.
$$
This shows that Theorem~\ref{High dim SM} is optimal.
\end{ex}

The next example shows that Theorem~\ref{High dim} 
does not hold when $B$ is not effective on the generic fiber. 
\begin{ex}
Consider $C\simeq \mathbb{P}^1$. Consider the pair $(C\times \mathbb{P}^1, B:=B_1-B_2)$ and the natural projection $\pi: C\times \mathbb{P}^1\to C$, where $B_1$ is a curve on $C\times \mathbb{P}^1$ of type $(2,3)$ with a cusp $q\in B_1$, and $B_2$ is the section of $\pi$ containing $q$. Denote $p=\pi(q)$ and $D=\pi^{-1}(p)=\pi^{*}p$. We can take $B_1,B_2$ so that $B_1$, $B_2$, and $D$ are locally defined by $(x^2+y^3=0)$, $(y=0)$, and $(x=0)$, respectively, for some local coordinates $x,y$ near $q\in C\times \mathbb{P}^1$. Then $\lct(C\times \mathbb{P}^1/C\ni p,B;D)=\frac{1}{3}<\frac{1}{2}$. More generally, if $B$ is not effective on the generic fiber, then there is no uniform lower bound for $\lct(C\times \mathbb{P}^1/C\ni p,B;D)$ as in Theorem~\ref{High dim}.
\end{ex}

\subsection{Proofs of Theorems~\ref{thm: Shokurov P1 fibration lct 1/2} and~\ref{High dim SM}}

We first reduce Theorem~\ref{High dim SM} to the case when $B$ is a $\Qq$-divisor.

\begin{lem}\label{the reduction to Q case}
Assume that Theorem~\ref{High dim SM} holds when $B$ is a $\Qq$-divisor, then Theorem~\ref{High dim SM} holds.
\end{lem}

\begin{proof}
Fix the choice of the Weil divisor $K_X$.
We may write $$K_X+B=\sum_{i=1}^m d_iD_i,$$ where $D_i$ are Cartier divisors on $X$ and $ d_1,\ldots,d_m$ are $\mathbb{Q}$-linearly independent real numbers. By \cite[Lemma 5.3]{HLS19}, $D_i$ is $\Rr$-Cartier and
$D_i\sim_{\Rr,Z}0$ for any $1\leq i\leq m$.

For a point $\mathbf{t}=(t_1, \dots, t_m)\in \mathbb{R}^m$, we denote 
$$B(\mathbf{t})=\sum_{i=1}^m t_iD_i-K_X.$$
Then for any $\mathbf{t}\in \mathbb{R}^m$,
$K_X+B(\mathbf{t})\sim_{\Rr,Z}0$.
Denote $\mathbf{d}=(d_1, \dots, d_m)$.

Take $f:Y\to X$ be a log resolution of $(X,B+\sum_{i=1}^m D_i)$ such that $\Supp f^{-1}\pi^{-1}(\overline{z})$ is a simple normal crossing divisor. Write $K_Y+B_Y(\mathbf{t})=f^{*}(K_X+B(\mathbf{t}))$. 

Possibly shrinking $Z$ near $z$, we may assume that $(X, B)$ is lc.
Note that $(X, B(\mathbf{t}))$ is lc if and only if the coefficients of $B_Y(\mathbf{t})$ are at most $1$. 
Note that $\mld(X/Z\ni z, B(\mathbf{t}))\geq 1$ if and only if for any prime divisor $E$ on $Y$ with $f(E)=\overline{z}$, $\mult_E B_Y(\mathbf{t})\leq 0$ (cf. \cite[Lemma 3.3]{CH21}). 
So the subset
$$\mathcal{P}_1:=\{\mathbf{t}\in \mathbb{R}^m\mid (X, B(\mathbf{t})) \text{ is lc}, \mld(X/Z\ni z, B(\mathbf{t}))\geq 1\}$$ is determined by finitely many linear functions in $\mathbf{t}$ with coefficients in $\mathbb{Q}$. In other words, 
$\mathcal{P}_1$ is a rational polytope containing $\mathbf{d}$. Note that $\mld(X/Z\ni z, B(\mathbf{t}))$ can be computed on $Y$ as the minimum of finitely many linear functions in $\mathbf{t}$ with coefficients in $\mathbb{Q}$, possibly replacing $\mathcal{P}_1$ with a smaller rational polytope containing $\mathbf{d}$, we may assume that $\mld(X/Z\ni z, B(\mathbf{t}))$ is linear on $\mathcal{P}_1$ and $\mathcal{P}_1$ is bounded.

By Remark~\ref{rem lc B>0}, $B$ is effective on the generic fiber of $\pi$. It is easy to see that $$\mathcal{P}_2:=\{\mathbf{t}\in \mathbb{R}^m\mid B(\mathbf{t}) \text{ is effective on the generic fiber of } \pi\}$$ is a rational polytope.

By the construction, $\mathcal{P}:=\mathcal{P}_1\cap \mathcal{P}_2$ is a bounded rational polytope containing $\mathbf{d}$. If $\mathbf{t}\in \mathcal{P}$, then $\pi: (X, B(\mathbf{t}))\to Z$ is an lc-trivial fibration satisfying Theorem~\ref{High dim SM}. So we can consider the canonical bundle formula $$K_X+B(\mathbf{t})=\pi^*(K_Z+B(\mathbf{t})_Z+M(\mathbf{t})_Z).$$ By the convexity of log canonical thresholds, irreducible components of $\Supp B(\mathbf{t})_Z$ belong to a finite set $\{P_1, P_2, \dots, P_k\}$ for any $\mathbf{t}\in \mathcal{P}$, here $\{P_1, P_2, \dots, P_k\}$ is the set of prime divisors on $Z$ in $\bigcup_{\mathbf{t}'}\Supp B(\mathbf{t}')_Z$ where the union runs over all vertex points $\mathbf{t}'\in \mathcal{P}$.
Denote the generic point of $P_j$ by $z_j$ for $1\leq j\leq k$.
Note that for any $1\leq j\leq k$, $\lct(X/Z\ni z_j, B(\mathbf{t}); \pi^*P_j)$ is computed on a log resolution as the minimum of finitely many linear functions in $\mathbf{t}$ with coefficients in $\mathbb{Q}$. So possibly replacing $\mathcal{P}$ with a smaller rational polytope containing $\mathbf{d}$, we may assume that $\lct(X/Z\ni z_j, B(\mathbf{t}); \pi^*P_j)$ is linear in $\mathbf{t}$ for any $1\leq j\leq k$.

Now we can take $\mathbf{t}_1, \dots, \mathbf{t}_l\in \mathcal{P} \cap \mathbb{Q}^m$ and positive real numbers $s_1,\dots, s_l$ such that $\sum_{i=1}^ls_i=1$ and $\sum_{i=1}^ls_i\mathbf{t}_i=\mathbf{d}$.
By the construction,
\begin{align*}
 B_Z={}&\sum_{j=1}^k (1-\lct(X/Z\ni z_j, B; \pi^*P_j))P_j\\
 ={}&\sum_{j=1}^k \sum_{i=1}^ls_i (1-\lct(X/Z\ni z_j, B(\mathbf{t}_i); \pi^*P_j))P_j=\sum_{i=1}^ls_i B(\mathbf{t}_i)_Z.
\end{align*}

By assumption, Theorem~\ref{High dim SM} holds for
$(X/Z\ni z, B(\mathbf{t}_i))$
for each $i$, that is, we can choose $M(\mathbf{t}_i)_Z\geq 0$ such that $$\mld(Z\ni z,B(\mathbf{t}_i)_Z+M(\mathbf{t}_i)_Z)\geq \mld(X/Z\ni z, B(\mathbf{t}_i))-\frac{1}{2}.$$
Then set $M_Z:=\sum_{i=1}^l s_i M(\mathbf{t}_i)_Z\geq 0$, we have
\begin{align*}
 &\mld(Z\ni z,B_Z+M_Z)\geq {} \sum_{i=1}^l s_i
 \mld(Z\ni z,B(\mathbf{t}_i)_Z+M(\mathbf{t}_i)_Z)\\
 \geq {}& \sum_{i=1}^l
 s_i\mld(X/Z\ni z, B(\mathbf{t}_i))-\frac{1}{2}
 ={}\mld(X/Z\ni z, B)-\frac{1}{2}.
\end{align*}
Here for the first inequality, we use the convexity of minimal log discrepancies, and for the last equality we use the linearity of $\mld(X/Z\ni z, B(\mathbf{t}))$ on $\mathcal{P}$.
\end{proof}

\begin{proof}[Proof of Theorem~\ref{High dim SM}]
By Lemma~\ref{the reduction to Q case}, we may assume that $B$ is a $\mathbb{Q}$-divisor.
As we described in Section~\ref{sec cbf}, there are b-divisors $\mathbf{B}$ and $\mathbf{M}$ such that
\begin{itemize}
 \item $\mathbf{B}_Z=B_Z$, $\mathbf{M}_Z=M_Z$, and
 \item for any birational contraction $g: Z'\to Z$, let $X'$ be a resolution of the main component of $X\times_{Z} Z'$ with induced morphisms $g': X' \to X$ and $\pi': X'\to Z'$. Write $K_{X'}+B'=g'^*(K_X+B)$, then $\mathbf{B}_{Z'}$ (respectively, $\mathbf{M}_{Z'}$) is the discriminant part (respectively, the moduli part) of the canonical bundle formula of $K_{X'}+B'$ on $Z'$.
\end{itemize}

We may write $\mathbf{B}=\sum d_P P$, where $P$ is the birational component of $\mathbf{B}$ and $d_P$ the corresponding coefficient. 
\begin{claim}\label{bound coefficient of birational divisor}For any birational component $P$ of $\mathbf{B}$ whose center on $Z$ is $\overline{z}$, $d_P\leq \frac{3}{2}-\mld(X/Z\ni z, B)$.
\end{claim}

We will proceed the proof assuming Claim~\ref{bound coefficient of birational divisor}. The proof of Claim~\ref{bound coefficient of birational divisor} will be given after the proof.

 By \cite[Theorem~8.1]{PS09} (see Remark~\ref{rem PS8.1}), $\mathbf{M}$ is b-semi-ample. Then 
 there exists a resolution $g:Z'\to Z$ such that $\mathbf{M}_{Z'}$ is semi-ample, and $\mathbf{B}_{Z'}+\Supp(g^{-1}(\overline{z}))$ is a simple normal crossing divisor.
Thus we may take a general $\mathbb{Q}$-divisor $L_{Z'}\geq 0$ on $Z'$ such that
 $\mathbf{M}_{Z'}\sim_{\Qq} L_{Z'}$,
 $\mathbf{B}_{Z'}+L_{Z'}$ is simple normal crossing, and for each prime divisor $P$ on $Z'$ whose center on $Z$ is $\overline{z}$, the coefficient of $P$ in $\mathbf{B}_{Z'}+L_{Z'}$ is at most $\frac{3}{2}-\mld(X/Z\ni z, B)$.
In this case, 
$
\mld(Z'/Z\ni z, \mathbf{B}_{Z'}+L_{Z'})\geq \mld(X/Z\ni z, B)-\frac{1}{2}.
$
Note that $$K_{Z'}+\mathbf{B}_{Z'}+L_{Z'}\sim_{\Qq}K_{Z'}+\mathbf{B}_{Z'}+\mathbf{M}_{Z'}=g^*(K_Z+B_Z+\mathbf{M}_Z)\sim_{\Qq, Z} 0,$$
hence by the negativity lemma \cite[Lemma~3.39]{KM98}, $$g^*(K_Z+B_Z+g_*L_{Z'})=g^*g_*(K_{Z'}+\mathbf{B}_{Z'}+L_{Z'})=K_{Z'}+\mathbf{B}_{Z'}+L_{Z'}.$$
Thus $ M_Z\sim_{\Qq} g_*L_{Z'}\geq 0$ and 
$
\mld(Z\ni z, B_{Z}+g_*L_{Z'})\geq \mld(X/Z\ni z, B)-\frac{1}{2}.$
\end{proof}

\begin{proof}[Proof of Claim~\ref{bound coefficient of birational divisor}]
Fix a birational component $P_0$ of $\mathbf{B}$ whose center on $Z$ is $\overline{z}$.
$$\xymatrix@R=2em{
 ({X'}, B') \ar[d]_{\pi'} \ar[r]^{g'} & (X, B) \ar[d]^{\pi}\\
Z' \ar[r]_{g} & Z
}
$$
Take a resolution $g: Z'\to Z$ such that $P_0$ is a prime divisor on $Z'$.
Denote the generic point of $P_0$ on $Z'$ to be $z'$ and hence $P_0=\overline{z'}$.
Let $X'$ be a resolution of the main component of $X\times_Z Z'$ 
with induced maps $g': {X'}\to X$ and $\pi': {X'}\to Z'$. 
We may write
$K_{X'}+B'=g'^*{(K_X+B)}$.
Then $$\mld({X'}/Z\ni z, B')=\mld(X/Z\ni z, B)\geq 1.$$ 
In particular, this implies that $$\mld({X'}/Z'\ni z', B')\geq \mld(X/Z\ni z, B) \geq 1.$$
By the construction, the generic fiber of $\pi'$ is a rational curve. So $({X'}/Z'\ni z', B')$ satisfies the assumptions of Theorem~\ref{High dim}.
By Theorem~\ref{High dim}, $$\lct({X'}/Z'\ni z',B';\pi'^*\overline{z'})\geq \mld({X'}/Z'\ni z', B')-\frac{1}{2}.$$
Hence by the definition of $\mathbf{B}$,
\begin{align*}
 d_{P_0}={}&1-\lct({X'}/Z'\ni z',B';\pi'^*\overline{z'})\\\leq{}& \frac{3}{2}-\mld({X'}/Z'\ni z', B')\leq \frac{3}{2}-\mld(X/Z\ni z, B).
\end{align*}
\end{proof}

\begin{proof}[Proof of Corollary~\ref{cor: High dim SM}]
This is directly by applying Theorem~\ref{High dim SM} to all comdimsion $\geq 1$ points on $Z$.
\end{proof}
\begin{proof}[Proof of Theorem~\ref{thm: Shokurov P1 fibration lct 1/2}] As the statement is local, we may assume that $Z$ is affine. Since $-K_X$ is ample over $Z$, there exists a positive integer $N$, such that $-NK_X$ is very ample over $Z$. Let $H$ be a general very ample divisor on $X$ such that $H\sim_{Z} -NK_X$ and take $B=\frac{1}{N}H$. Then $K_X+B\sim_{\Qq,Z}0$, $B$ has no vertical irreducible component over $Z$, and $(X,B)$ is canonical.
By Corollary~\ref{cor: High dim SM}, we can choose $M_Z\geq 0$ representing the moduli part and $B_Z$ the discriminant part of the canonical bundle formula of $K_X+B$ on $Z$, so that $(Z,B_Z+M_Z)$ is $\frac{1}{2}$-lc. Note that $B\geq 0$ implies that $B_Z\geq 0$. Thus $Z$ is $\frac{1}{2}$-lc.
\end{proof}

Finally, as an application of Corollary~\ref{cor: High dim SM}, we show the following weaker version of Iskovskikh's conjecture under more general setting without using the classification of terminal singularities in dimension $3$ as in \cite{MP08}.

\begin{cor}\label{cor weak isk}
Let $\pi:X\to Z$ be a contraction between normal varieties, such that
\begin{enumerate}
 \item $\dim X-\dim Z=1$, 
 \item\label{4.8 condition 2} there is no prime divisor $D$ on $X$ such that $\text{\rm codim}(\pi(D), Z)\geq 2$,
 \item\label{4.8 condition 3} $X$ is terminal,
 \item $K_Z$ is $\mathbb{Q}$-Cartier, and
 \item\label{4.8 condition 5} $-K_X$ is ample over $Z$.
\end{enumerate}
Then $Z$ is $\frac{1}{2}$-klt.
\end{cor}

Here assumption~\eqref{4.8 condition 2} is a natural geometric condition, for example, it holds if all fiber of $\pi$ are 1-dimensional or if $\rho(X/Z)=1$.
\begin{proof}
As the statement is local, we may assume that $Z$ is affine. By Theorem~\ref{thm: Shokurov P1 fibration lct 1/2}, $Z$ is $\frac{1}{2}$-lc.
Assume to the contrary that $Z$ is not $\frac{1}{2}$-klt, then there exists an exceptional prime divisor $E$ over $Z$ such that $a(E, Z )=\frac{1}{2}$. Denote by $c_Z(E)$ the center of $E$ on $Z$.

By \cite[Corollary 1.4.3]{BCHM10}, we can find a proper birational morphism $g: Z'\to Z$ such that $E$ is the only $g$-exceptional divisor. Let $X'$ be a resolution of the main component of $X\times_{Z} Z'$ with induced morphisms $g': X' \to X$ and $\pi': X'\to Z'$. 
 $$\xymatrix@R=2em{X' \ar[d]_{\pi'} \ar[r]^{g'} & X \ar[d]_{\pi}\\
 Z' \ar[r]_{g} & Z
 }
 $$
We can write $K_{X'}+G=g'^*K_X$, $K_{Z'}+\frac{1}{2}E=g^*K_Z.$

As $-K_X$ is ample over $Z$, for $t\in (0, 1)$, we can take an effective $\mathbb{Q}$-divisor $B^t$ on $X$ such that 
\begin{itemize}
\item $(X,B^t)$ is canonical, 
\item $B^t$ has no vertical irreducible component over $Z$, 
\item $K_X+B^t\sim_{\mathbb{Q},Z} 0$, and
\item $\Supp B^t\supset \Supp(\pi^{-1}(c_Z(E)))$, and the multiplicity of each irreducible component of $\Supp(\pi^{-1}(c_Z(E)))$ in $B^t$ is a non-constant linear function in $t$. 
\end{itemize}
The construction is as follows. Take a sufficiently large $N$ such that $-NK_X\sim_{Z} H$ is a very ample divisor on $X$, and $\mathcal{O}_X(H)$ and $\mathcal{O}_X(H)\otimes I_{\Supp(\pi^{-1}(c_Z(E)))}$ are generated by global sections. Now take $B_1$ to be a general global section of $\mathcal{O}_X(H)$ and $B_2$ a general global section of $\mathcal{O}_X(H)\otimes I_{\Supp(\pi^{-1}(c_Z(E)))}$. Then $B^t=\frac{(1-st)}{N}B_1+\frac{st}{N} B_2$ satisfies the requirements for sufficiently small positive rational number $s$. Here assumption~\eqref{4.8 condition 3} guarantees that $(X, B^t)$ is canonical, and assumption~\eqref{4.8 condition 2} guarantees that $B^t$ has no vertical irreducible component over $Z$ as $\Supp(\pi^{-1}(c_Z(E)))$ has codimension at least $2$ in $X$.

Then by Corollary~\ref{cor: High dim SM}, for $t\in (0, 1)$, we can choose $M^t_Z\geq 0$ representing the moduli part of the canonical bundle formula of $K_X+B^t$ on $Z$, so that $(Z,B^t_Z+M^t_Z)$ is $\frac{1}{2}$-lc, where $B^t_Z\geq 0$ is the discriminant part. In particular, $c_Z(E)$ is not contained in $\Supp(B^t_Z+M^t_Z)$.
As we described in Section~\ref{sec cbf}, there are b-divisors $\mathbf{B}^t$ and $\mathbf{M}^t$ such that
\begin{itemize}
 \item $\mathbf{B}^t_Z=B^t_Z$, $\mathbf{M}^t_Z=M^t_Z$, 
 \item $K_{X'}+G+g'^*B^t=\pi'^*(K_{Z'}+\mathbf{B}^t_{Z'}+\mathbf{M}^t_{Z'})$,
 \item $K_{Z'}+\mathbf{B}^t_{Z'}+\mathbf{M}^t_{Z'}=g^*(K_Z+B^t_Z+M^t_Z)=K_{Z'}+\frac{1}{2}E+g^*(B^t_Z+M^t_Z).$
 \end{itemize}
 Recall that $\mathbf{M}^t$ is b-semi-ample by \cite[Theorem~8.1]{PS09} (see Remark~\ref{rem PS8.1}), so $\mathbf{M}^t_{Z'}\leq g^*M^t_Z$ by the negativity lemma \cite[Lemma~3.39]{KM98}.
 As $c_Z(E)$ is not contained in $\Supp(B^t_Z+M^t_Z)$, we get $\mult_E \mathbf{M}^t_{Z'}=0$ and then $\mult_E\mathbf{B}^t_{Z'}=\frac{1}{2}$.
The latter one implies that $\lct(X'/Z'\ni \eta_E, G+g'^*B^t; \pi'^*E)=\frac{1}{2}$ by definition, where $\eta_E$ is the generic point of $E$.
This is absurd, as by the construction of $B^t$, $\lct(X'/Z'\ni \eta_E, G+g'^*B^t; \pi'^*E)$ is a non-constant function in $t$. \end{proof}

\begin{rem}
\begin{enumerate}

 \item By Example~\ref{ex: counterex of 1/2-klt}, assumption~\eqref{4.8 condition 3} of Corollary~\ref{cor weak isk} can not be replaced by ``$X$ is canonical".
 
 \item We expect that assumptions~\eqref{4.8 condition 2} and~\eqref{4.8 condition 5} of Corollary~\ref{cor weak isk} are all necessary. In fact, by the terminalization of Example~\ref{ex: counterex of 1/2-klt}, assumptions~\eqref{4.8 condition 2} and~\eqref{4.8 condition 5} can not be removed at the same time.
\end{enumerate}

\end{rem}

Prokhorov provides us the following example, which shows that Corollary~\ref{cor weak isk} can not be improved if $\dim X\geq 4.$ 

\begin{ex}\label{ex: counterex 4.8}
Consider the following action of $\bmu_{2m+1}$ on $\mathbb{P}_x^1\times \mathbb{C}^3_{u, v, w}$:
$$
(x; u, v, w)\mapsto (\xi^m x; \xi u, \xi v, \xi^{m} w),
$$
where $m$ is a positive integer and $\xi$ is a primitive $(2m+1)$-th root of unity. Let $X=(\mathbb{P}^1\times \mathbb{C}^3)/\bmu_{2m+1}$, $Z= \mathbb{C}^3/\bmu_{2m+1}$, and $\pi: X\to Z$ the natural projection. Since $\bmu_{2m+1}$ acts freely in codimension $1$, $-K_X$ is $\pi$-ample and $\rho(X/Z)=1$. Note that $Z$ has an isolated cyclic quotient singularity of type $\frac{1}{2m+1}(1, 1, m)$ at the origin $o\in Z$, and $\mld(Z\ni o)=\frac{m+2}{2m+1}$ (see \cite{Amb06} for the computation of minimal log discrepancies of toric varieties). On the other hand, $X$ has two isolated cyclic quotient singularity of types $\frac{1}{2m+1}(m, 1, 1, m)$ and $\frac{1}{2m+1}(m+1, 1, 1, m)$, which are terminal (see \cite[(4.11)~Theorem]{YPG}). 
\end{ex}

\appendix
\section{Bounding log canonical thresholds by cyclic coverings}\label{sec appendix}

In this appendix, we will prove Theorem~\ref{Smooth local case general alpha version}, a weaker version of Theorem~\ref{Smooth local case general beta version}, by a different method. Although the result is weaker, the advantage is that we do not use the convexity to reduce to the case that $B$ is irreducible, instead we use a covering trick.

\begin{thm}\label{Smooth local case general alpha version}
Let $(X\ni P, B)$ be a germ of surface pair such that $X$ is smooth and $\mult_P B\leq 1$. Let $C$ be a smooth curve at $P$ such that $C\nsubseteq\Supp B$.
Denote $\mult_P B=m$, $(B\cdot C)_P=I$.
Suppose that $\frac{m}{I}\geq m-\frac{1}{2}$.
Then $\lct(X\ni P,B;C)\geq \min\{1, 1-m+\frac{m}{I}\}$.
\end{thm}

\begin{defn}
For an effective $\mathbb{Q}$-divisor $B$ on a smooth formal surface germ $P\in X$ with local coordinate systems $(x, y)$, suppose that we have an expression $B=\sum_{i=1}^k b_i B_i$ where $B_i$ are Cartier divisors defined by equations $(f_i=0)$ for $f_i(x,y)\in \mathbb{C}[[x,y]]$, by abusing the notation, we say that $(f=0)$ is the equation of $B$ where $f=\prod_{i=1}^{k}f_i(x, y)^{b_i}$. Given weights $w(x)$ and $w(y)$, we define $w (f_i )$ to be the weight of the lowest weight term of $f_i$, and define $w(f)=\sum_{i=1}^k b_iw (f_i )$. Note that $w(f)$ does not depend on the choice of expressions of $B$. We say $w(f)$ is the weight of $f$ with respect to $w(x),w(y)$.
\end{defn}

Here we recall an equivariant version of a theorem due to Var\v{c}enko on computing log canonical thresholds on a smooth formal surface germ.
\begin{thm}[{\cite{Var76}, \cite[Theorem~6.40]{KSC04}}]\label{varcenko theorem}
Let $P\in X$ be a smooth formal surface germ. Let $G$ be a finite Abelian group acting on $P\in X$ and 
let $B$ be an effective $G$-invariant $\mathbb{Q}$-divisor on $X$. 
Then 
$$\lct(X\ni P; B)=\inf_{x,y, w}\frac{w(x) + w(y)}{w(f)}
$$
where the infimum runs over all $G$-invariant local coordinate systems $(x, y)$ for $P\in X$ and over all choices of weights $w(x)$ and $w(y)$ (positive integers), and where $(f = 0)$ is the equation of the $\mathbb{Q}$-divisor $B$ in the coordinates $x, y$. Here a local coordinate system $(x, y)$ is {\it $G$-invariant} if $(x=0)$ and $(y=0)$ are $G$-invariant.
\end{thm}
\begin{proof}
If $B$ is a Cartier divisor and $G$ is trivial, then this is exactly \cite[Theorem~6.40]{KSC04}.
In general, 
if $B$ is a $\mathbb{Q}$-divisor, then we may assume that $mB$ is Cartier for some positive integer $m$. Hence by \cite[Theorem~6.40]{KSC04}, 
$$
 \lct(X\ni P; B)=m\lct(X\ni P; mB)= m\inf_{x,y,w}\frac{w(x) + w(y)}{w(f^m)}= \inf_{x,y,w}\frac{w(x) + w(y)}{w(f)},
$$
where the infimum runs over all local coordinate systems $(x, y)$ for $P\in X$ and over all choices of weights $w(x)$ and $w(y)$, and where $(f = 0)$ is the equation of the $\mathbb{Q}$-divisor $B$ in the coordinates $x, y$.

Note that in the above formula, we need to consider all local coordinate systems instead of $G$-invariant ones.
So to conclude the proof, we only need to show that $\lct(X\ni P; B)$ is computed by a weighted blow-up in a suitable $G$-invariant local coordinate system, that is, denote $t=\lct(X\ni P; B)$, then there exists a weighted blow-up $\pi: Y\to X$ at $P$ in a suitable $G$-invariant local coordinate system and a $G$-invariant prime divisor $E$ on $Y$ with $P\in \pi(E)$ such that $a(E, X, tB)=0$. Here it is possible that $\pi$ is the identity map and $E$ is a prime divisor on $X$. 

Take $Z$ to be the minimal non-klt center of $(X, tB)$ containing $P$. Then $Z$ is $G$-invariant by the minimality. By Proposition~\ref{prop unique E}, there exists a $G$-invariant effective $\mathbb{Q}$-Cartier divisor $B'$ such that 
\begin{enumerate}
\item $Z$ is an isolated non-klt center of $(X, B')$, and
\item there is a unique prime divisor $E$ over $X$ with center $Z$, such that $a(E,X,B')=0$. Moreover, $a(E,X,tB)=0$.
\end{enumerate} 
In other words, $E$ is the unique prime divisor over $X$ computing $\lct(X\ni P; B')$. 
By the proof of \cite[Theorem~6.40]{KSC04} (or \cite[Theorem~1]{Kaw17}), $E$ can be obtained by a weighted blow-up $\pi: Y\to X$. By the uniqueness of $E$, $\pi$ is $G$-equivariant.

Finally we show that the $G$-equivariant weighted blow-up $\pi$ is indeed a weighted blow-up in a suitable $G$-invariant local coordinate system following \cite[Theorem~1]{Kaw17}. Let $\mathfrak{m}$ be the maximal ideal of $\mathcal{O}_{X, P}$. As $G$ is Abelian, for any $k\geq 2$,
consider the subset $V_k \subseteq\mathfrak{m}/\mathfrak{m}^k$ consisting of all common eigenvectors of $g^*$-action for all $g\in G$, then $\text{Span}_\mathbb{C}(V_k)=\mathfrak{m}/\mathfrak{m}^k$ and $\{V_k\}$ forms an inverse system. By considering $\varprojlim V_k$, we can find a local coordinate system $(x, y)$ such that for any $g\in G$, $g^*(x)=\lambda_g x$ and $g^*(y)=\mu_g y$ for some $\lambda_g, \mu_g\in \mathbb{C}^*$.
Consider 
$$
a=\min_{t\in \mathfrak{m}\setminus \mathfrak{m}^2}\text{ord}_E (t)
\ \ \ \text{and}\ \ \ 
b=\max_{t\in \mathfrak{m}\setminus \mathfrak{m}^2}\text{ord}_E (t),
$$
where $\text{ord}_E$ is the divisorial valuation defined by $E$.
Recall that by the proof of \cite[Theorem~1]{Kaw17}, if $\text{ord}_E(x_1)=a$ and $
\text{ord}_E(y_1)=b$ for some $x_1, y_1\in \mathfrak{m}\setminus \mathfrak{m}^2$, then $\pi$ is a weighted blow-up in the local coordinate system $(x_1, y_1)$.
Fix a suitable local coordinate system $(x', y')$ with $\text{ord}_E(x')=a$ and $
\text{ord}_E(y')=b$. 
As $E$ is $G$-invariant, for any $g\in G$, $\text{ord}_E(g^*(x'))=a$ and $
\text{ord}_E(g^*(y'))=b$.
Suppose that 
$y'-ux-vy\in \mathfrak{m}^2$ for some $u, v\in \mathbb{C}$. Without loss of generality, we may assume that $v\neq 0$.
Now consider $$y_0=\sum_{g\in G}\frac{1}{\mu_g}g^*(y').$$
It is easy to check that $y_0\in \mathfrak{m}\setminus \mathfrak{m}^2$ and
$g^*(y_0)=\mu_g y_0$ for all $g\in G$. Also we have $$\text{ord}_E(y_0)\geq \min\{\text{ord}_E(g^*(y'))\mid g\in G\}=b.$$
So $\text{ord}_E(y_0)=b$ by the maximality of $b$.
On the other hand, it is easy to see that $$\min\{\text{ord}_E(x), \text{ord}_E(y)\}= \min\{\text{ord}_E(x'), \text{ord}_E(y')\}=a, $$
as $\text{ord}_E$ is the monomial valuation determined by $\text{ord}_E(x')$ and $\text{ord}_E(y')$.
So either $\text{ord}_E(x)=a$ or $\text{ord}_E(y)=a$. 
Then by the proof of \cite[Theorem~1]{Kaw17}, $\pi$ is a weighted blow-up in the $G$-invariant local coordinate system $(x, y_0)$ or $(y, y_0)$ with weight $(a, b)$.
\end{proof}

Here we recall the equivariant version of the ``tie breaking" method.
\begin{prop}[{cf. \cite[Proposition~8.7.1]{Cor07}}]\label{prop unique E}
Let $X$ be a quasi-projective variety or a quasi-projective non-singular $R$-variety over a formal power series ring $R$.
Let $G$ be a finite group acting on $X$. Let $(X, \Delta_1)$ be a $G$-invariant klt pair and $\Delta_2$ a $G$-invariant effective $\mathbb{Q}$-Cartier divisor such that $(X, \Delta_1+\Delta_2)$ is lc. Let $Z \subseteq X$ be a minimal non-klt center of $(X, \Delta_1+\Delta_2)$ which is $G$-invariant. Then there exists a $G$-invariant effective $\mathbb{Q}$-Cartier divisor $\Delta'_2$ such that 
\begin{enumerate}
\item $Z$ is an isolated non-klt center of $(X, \Delta_1+\Delta'_2)$, and
\item there is a unique non-klt place of $(X, \Delta_1+\Delta'_2)$ with center $Z$, and it is also a non-klt place of $(X, \Delta_1+\Delta_2)$.
\end{enumerate} 
\end{prop}

\begin{proof}
When $G$ is trivial this is \cite[Proposition~8.7.1]{Cor07}. 
We explain how to modify the proof of \cite[Proposition~8.7.1]{Cor07}.
Take $D$ to be an effective $G$-invariant divisor on $X$ such that $Z$ is the only non-klt center contained in $D$. Let $f: X'\to X$ be a $G$-equivariant log resolution of $(X, \Delta_1+\Delta_2+D)$ (\cite{AW97, Tem08}). Then by the proof of \cite[Proposition~8.7.1]{Cor07}, we can find positive rational numbers $\epsilon<1$ and $\eta$ such that $Z$ is the only non-klt center of $(X, \Delta_1+(1-\epsilon)\Delta_2+\eta D)$ and all the non-klt places are also non-klt places of $(X, \Delta_1+\Delta_2)$. Denote $\Delta_3= (1-\epsilon)\Delta_2+\eta D$.

Then by the proof of \cite[Proposition~8.7.1]{Cor07}, we can find an effective $G$-invariant divisor $D'$ on $X$ such that $f^*D'$ is simple normal crossing and there is one irreducible component $E_a$ of $f^*D'$ which is ample on $X'$. By the proof of \cite[Proposition~8.7.1]{Cor07}, we can find positive rational numbers $\epsilon'<1$ and $\eta'$ such that $Z$ is the only non-klt center of $(X, \Delta_1+(1-\epsilon')\Delta_3+\eta' D')$ and all the non-klt places are also non-klt places of $(X, \Delta_1+\Delta_2)$.

Write $K_{X'}+\Delta'=f^*(K_X+\Delta_1+(1-\epsilon')\Delta_3+\eta' D')$. Take $T$ to be a minimal non-klt center of $({X'}, \Delta')$, then for any $g\in G$, $g(T)\cap T$ is either $\emptyset$ or $T$.
Blowing up $X'$ along $\cup_{g\in G}g(T)$, we get a birational morphism $f': X''\to X'$ and a new $G$-equivariant log resolution $f'': X''\to X$ with a divisor $E_0=f'^{-1}(T)$ such that $E_0$ is a non-klt place of $(X, \Delta_1+(1-\epsilon')\Delta_3+\eta' D')$ and
$g(E_0)\cap E_0$ is either $\emptyset$ or $E_0$ for any $g\in G$. Now note that $\sum_{g\in G}(f'^*g(E_a)-tg(E_0))$ is ample for sufficiently small positive number $t$. So by the proof of \cite[Proposition~8.7.1]{Cor07}, we can perturb the coefficients of $f''^*D'$ to get an effective $G$-invariant divisor $D''\sim_{\mathbb{Q}}D'$ so that there are positive rational numbers $\epsilon''<1$ and $\eta''$ such that $Z$ is the only non-klt center of $(X, \Delta_1+(1-\epsilon'')\Delta_3+\eta'' D'')$ and $\{g(E_0)\mid g\in G\}$ is the set of non-klt places of $(X, \Delta_1+(1-\epsilon'')\Delta_3+\eta'' D'')$.
Now by the connectedness lemma (\cite[Theorem~5.48]{KM98}, \cite[Theorem~3.1]{Kaw15}), $\cup_{g\in G}g(E_0)\to Z$ has connected geometric fibers, which means that $\{g(E_0)\mid g\in G\}=\{E_0\}$ consists of a unique non-klt place, which is also a non-klt place of $(X, \Delta_1+\Delta_2)$ by the construction. 
\end{proof}

\begin{proof}[Proof of Theorem~\ref{Smooth local case general alpha version}]
If $I\leq 1$, then $(X\ni P,B+C)$ is lc by \cite[Corollary~5.57]{KM98}. Hence we may assume that $I>1$.
By assumption, $0<m-\frac{m}{I}\leq \frac{1}{2}$. We may take an integer $n\geq 2$ and a real number $0\leq \epsilon<1$ such that $m-\frac{m}{I}=\frac{1-\epsilon}{n}$.
We need to show that $(X\ni P, {B}+(1-\frac{1}{n}+\frac{\epsilon}{n}) C)$ is lc.
As being lc is a closed condition for coefficients, by slightly modifying the coefficients of $B$, we may assume that $B$ is a $\mathbb{Q}$-divisor and $\epsilon$ is a rational number.
We may assume that $(P\in X)\simeq (o\in \widehat{\mathbb{C}^2})$ is the formal neighborhood with coordinates ${x, y}$ and $C=(x=0)$. 

Consider the finite covering $\mu: \widehat{\mathbb{C}^2}\to \widehat{\mathbb{C}^2}$ defined by $(x, y)\mapsto (x^n, y )$ of degree $n$ ramified along $C$. Then 
$$
K_{\widehat{\mathbb{C}^2}}+\mu^*{B}+\epsilon C= \mu^*\left(K_{\widehat{\mathbb{C}^2}}+{B}+(1-\frac{1}{n}+\frac{\epsilon}{n}) C\right).
$$
By \cite[Proposition~5.20]{KM98}, $({\widehat{\mathbb{C}^2}}\ni o, {B}+(1-\frac{1}{n}+\frac{\epsilon}{n}) C)$ is lc if and only if $({\widehat{\mathbb{C}^2}}\ni o, \mu^*{B}+\epsilon C)$ is lc. Here the pair $({\widehat{\mathbb{C}^2}}, \mu^*{B}+\epsilon C)$ admits a natural $(\mathbb{Z}/n\mathbb{Z})$-action induced by $\mu$.
We will apply Theorem~\ref{varcenko theorem} to show that $({\widehat{\mathbb{C}^2}}\ni o, \mu^*{B}+\epsilon C)$ is lc, or equivalently, $\lct({\widehat{\mathbb{C}^2}}\ni o; \mu^*{B}+\epsilon C)\geq 1$.

Fix a $(\mathbb{Z}/n\mathbb{Z})$-invariant local coordinate system $(x', y')$ and fix weights $w(x'), w(y')$. Possibly switching $x'$ and $y'$ and rescaling, we may write $x'=x(1+h_1(x^n, y))$ and $y'=y+h_2(x^n, y)$ for some $h_1, h_2\in \mathbb{C}[[x, y]]$ with $\mult_o h_1(x^n, y)\geq 1$ and $\mult_o h_2(x^n, y)\geq 2$. As $1+h_1(x^n, y)$ is a unit in $\mathbb{C}[[x, y]]$, there exists a unit $u\in \mathbb{C}[[x', y']]$ such that $x=ux'$ and $w(x)=w(x')$ where $w(x)$ is the weight of $x$ with respect to $w(x'),w(y')$. 

Let $y_0=y+h_2(x, y)$. Then $\mu^*y_0=y'$ and $(x, y_0)$ is a local coordinate system. 
Suppose that $B=\sum_{i=1}^k b_i B_i$ for irreducible divisors $B_i$, and the equation of $B$ in the coordinates $(x, y_0)$ is of the form $\prod_{i=1}^k f_{i}(x,y_0)^{b_i}=0$. By \cite[\S 1, Exercise~5.14]{GTM52}, we may write $f_i(x,y_0)=(\alpha_i x+\beta_i y_0)^{m_i}+g_i$ for some $(\alpha_i, \beta_i)\neq (0,0)$, $m_i\in \mathbb{Z}_{>0}$, and $\mult_o g_i>m_i$. Denote by $I_i$ the minimal $k$ such that $y_0^k$ has non-zero coefficient in $f_i(x, y_0)$. Note that $I_i$ is well-defined as $C\not \subseteq \Supp B$, and we have $I_i\geq m_i$.
Then $$I=(B\cdot C)_o=\sum_{i=1}^kb_i I_i , \quad m=\mult_o B=\sum_{i=1}^kb_i m_i. $$
By assumption, $\sum_{i=1}^kb_i m_i\leq 1$, and 
\begin{align*}
 \left(1-\frac{1}{I}\right)\sum_{i=1}^kb_i m_i{}=\frac{ 1-\epsilon}{n}.
\end{align*}
In this setting, the equation of $\mu^{*}B+\epsilon C$ in $(x, y')$ is $(f=0)$ where $$f(x, y')=x^\epsilon\prod_{i=1}^k f_{i}(x^n,y')^{b_i}.$$
Let $w(f_i(x^n, y' ))$ be the weight of $f_i(x^n, y' )$ with respect to $w(x),w(y')$. Note that for each $i$, $$w(f_i(x^n, y'))\leq \begin{cases}\min\{nm_iw(x), I_iw(y')\} & \text{if } \alpha_i\neq 0;\\
m_iw(y') & \text{if } \alpha_i= 0.
\end{cases}$$

Write 
$g(x',y')=f(ux',y')$. Then the equation of $\mu^{*}B+\epsilon C$ in $(x', y')$ is $(g=0)$. Let $w(f)$ be the weight of $f(x,y')$ with respect to $w(x),w(y')$, and let $w(g)$ be the weight of $g(x',y')$ with respect to $w(x'),w(y')$.

If $nw(x)\leq w(y')$, then $w(f_i(x^n, y' ))\leq m_iw(y')$. Hence by Lemma~\ref{lem: replace weight},
$$
w(g)=w (f)\leq \epsilon w(x)+\sum_{i=1}^k b_i m_iw(y')\leq \epsilon w(x)+ w(y')
\leq w(x')+ w(y').$$

If $nw(x)> w(y')$, then $w(f_i(x^n, y' ))\leq (1-\frac{1}{I})nm_iw(x)+ \frac{1}{I}I_iw(y') $. Hence by Lemma~\ref{lem: replace weight},
$$
w(g)=w(f)\leq \epsilon w(x)+\sum_{i=1}^k b_i\left(\left(1-\frac{1}{I}\right)nm_iw(x)+ \frac{1}{I}I_iw(y') \right)= w(x')+ w(y').
$$

Hence by Theorem~\ref{varcenko theorem}, $({\widehat{\mathbb{C}^2}}\ni o, \mu^*{B}+\epsilon C)$ is lc.
\end{proof}

The following lemma is elementary on change of coordinates. We omit the proof.

\begin{lem}\label{lem: replace weight}
Suppose that $f(x,y')\in \mathbb{C}[[x,y']]$ and $x=ux'$ for some unit $u\in \mathbb{C}[[x',y']]$. Fix weights $w(x'),w(y')$. Consider $g(x',y')=f(ux',y')\in \mathbb{C}[[x',y']]$. Then $w(x)=w(x')$ and $w(f)=w(g)$, where $w(x)$ is the weight of $x$ with respect to $w(x'),w(y')$, $w(f)$ is the weight of $f(x,y')$ with respect to $w(x),w(y')$, and $w(g)$ is the weight of $g(x',y')$ with respect to $w(x'),w(y')$.
\end{lem}


\section{Bounding log canonical thresholds by Newton polytopes}\label{appendix: B second proof}
In this appendix, we will provide a self-contained proof of Corollary~\ref{cor lct B C 1}(a)(c), and thus Theorem~\ref{Smooth local case general beta version}, using Newton polytopes. The proof is inspired by \cite{Var76, KSC04, Col18}.

\begin{defn}
Let $f(x,y)\in\mathbb{C}[[x,y]]$ be a non-zero formal power series, we may write $$f(x,y)=\sum_{(p,q)\in\mathbb{Z}^2_{\geq 0}}a_{pq}x^p y^q.$$ 
\begin{enumerate}
 \item The \emph{Newton polytope} of $f$, denoted by $\NP(f)$, is the convex hull of $$\bigcup_{a_{pq}\neq 0}\left((p,q)+\mathbb{R}^2_{\geq 0}\right)$$
in $\mathbb{R}^2_{\geq 0}.$

\item 
The \emph{Newton distance} of $f$ 
is defined by $$\ndist{(f)}:=\sup\{t\in \mathbb{R}_{> 0}\mid (1,1)\in t\cdot \NP(f)\}.$$

\item \label{the main face and data}
The \emph{main face} $\mf(f)$ of $f$ is defined to be the minimal face of $\NP(f)$ containing $(\ndist{(f)}^{-1}, \ndist{(f)}^{-1})$.
Then $\mf(f)$ is either a 1-dimensional face or a vertex on the boundary of $\NP(f)$.
We define the \emph{Newton multiplicity} $\nmul(f)$ in the following way:
\begin{itemize}
 \item If either $\mf(f)$ is a vertex or $\mf(f)$ is not compact, then we define $\nmul(f)=\ndist{(f)}^{-1}$.
 \item If $\mf(f)$ is compact of dimension $1$ and its two vertices are denoted by $(p_1, q_1)$ and $(p_2, q_2)$ with $p_1<p_2$, then we define $\nmul(f):=\gcd(p_2-p_1, q_1-q_2)$.
 \end{itemize}
 
\item For a Cartier divisor $D$ on $\widehat{\mathbb{C}^2}$, choose local coordinates $(x,y)$ at $o\in\mathbb{C}^2$ and suppose that ${D}$ is defined by $(f=0)$ for some $f\in\mathbb{C}[[x,y]]$. We define the Newton polytope of ${D}$ to be $\NP({D}):=\NP(f)$, similarly we define $\ndist{({D})}:=\ndist{(f)}$ and $\nmul({D}):=\nmul(f)$. Note that all above definitions do not depend on the choice of $f$ up to a unit in $\mathbb{C}[[x,y]]$. However, all above definitions do depend on the choice of the coordinates $(x,y)$.
\end{enumerate}
\end{defn}

We collect some easy facts on weighted blow-ups.

\begin{lem}\label{lem wt blowup}
Let $a_1, a_2$ be two coprime positive integers.
Let $\pi: Y\to \mathbb{C}^2$ be the weighted blow-up at $o$ with coordinates $(x, y)$ and weight $(a_1, a_2)$. Then $Y\subset \mathbb{C}^2_{x,y}\times \mathbb{P}^1_{z, w}$ is defined by $(x^{a_2}w=y^{a_1}z)$ and the exceptional divisor $E\cong \mathbb{P}^1_{z, w}$. Denote by $D_1$ and $D_2$ the divisors on $\mathbb{C}^2$ defined by $(x=0)$ and $(y=0)$ respectively, and denote $D'_1$, $D'_2$ the strict transforms on $Y$. Denote $P_1=[0:1]$ and $P_2=[1:0]$ on $E$. Then
\begin{enumerate}
 \item $\pi^*K_{\mathbb{C}^2}=K_Y+(1-a_1-a_2)E$;
 \item $\pi^*D_i=D'_i+a_iE$ for $i=1, 2$;
 \item $(K_Y+E)|_E=K_E+(1-\frac{1}{a_2})P_1+(1-\frac{1}{a_1})P_2$;
 \item Suppose that $D$ is a divisor on $\mathbb{C}^2$ defined by $(f=0)$, take $f_w$ to be the weighted homogenous leading term of $f$, then we may write $f_w(x, y)=x^sy^th(x^{a_2}, y^{a_1})$ for some homogeneous polynomial $h$ of degree $d$. Denote $D'$ to be the strict transform of $D$ on $Y$. Then $\pi^*D=D'+(sa_1+ta_2+a_1a_2d)E$, and 
 $$D'|_E=\frac{s}{a_2}P_1+\frac{t}{a_1}P_2+G$$
 where $G$ is defined by $(h(z, w)=0)$ on $E$.
\end{enumerate}
 
\end{lem}

\begin{proof}
(1) and (2) are from \cite[Lemma 3.2.1]{Pro01}.
For (3), note that local computation (or toric geometry) gives
$$
(K_Y+D'_1+D'_2+E)|_E=K_E+P_1+P_2.
$$
On the other hand, by (2) and $-E^2=\frac{1}{a_1a_2}$, we have $D'_1|_E=\frac{1}{a_2}P_1$ and $D'_2|_E=\frac{1}{a_1}P_2$. This implies (3). (4) is by direct computation.
\end{proof}

\begin{lem}\label{the criterion lemma for attaining lct}
For any Cartier divisor ${D}$ on $\widehat{\mathbb{C}^2}$ with coordinates $(x,y)$, $$\ndist{({D})}\geq \lct(\widehat{\mathbb{C}^2}\ni o;{D})\ge \min\{\frac{1}{\nmul({D})},\ndist{({D})}\}$$ In particular, if $\ndist({D})\nmul({D})\leq 1$, then $\lct(\widehat{\mathbb{C}^2}\ni o;{D})=\ndist{({D})}$.

\end{lem}

\begin{proof}
Suppose that ${D}$ is defined by $(f=0)$ for some $f\in\mathbb{C}[[x,y]]$.
By 
{\cite[Theorem 32]{Kol08} or \cite[Proposition 2.5]{dFM09}},
there exists a positive integer $N$ such that for any $\widetilde{f}\in \mathbb{C}[x,y]$ such that $\mult_o(f-\widetilde{f})\geq N$, we have $$\lct(\mathbb{C}^2\ni o;\widetilde{f})=\lct(\widehat{\mathbb{C}^2}\ni o; \widetilde{f})=\lct (\widehat{\mathbb{C}^2}\ni o; f).$$ We can take such $\widetilde{f}$ so that $\NP(\widetilde{f})=\NP(f)$, and it suffices to prove the claim for the Cartier divisor $D$ defined by $(\widetilde{f}=0)$ near $o\in \mathbb{C}^2$. Thus by replacing $f$ with $\widetilde{f}$, we may assume that $f\in \mathbb{C}[x,y]$ and treat $\lct({\mathbb{C}^2}\ni o;{D})$.

If $\mf(D)$ is not compact, then possibly switching $x$ and $y$ and rescaling, we may assume that $f=x^{\ndist{(D)}^{-1}}(y^b+xh(x,y))$ for some $h\in \mathbb{C}[x,y]$ such that $b\leq \ndist{(D)}^{-1}$. Note that $\ndist{(D)}D=C_1+\ndist{(D)}C_2$, where $C_1$ is defined by $(x=0)$ and $C_2$ is defined by $(y^b+xh(x,y)=0)$. Note that $(C_1\cdot \ndist{(D)}C_2)=b\ndist{(D)}\leq 1.$
Then $(\mathbb{C}^2\ni o,\ndist{(D)}D)$ is lc by \cite[Corollary~5.57]{KM98}. So 
$\lct(\mathbb{C}^2\ni o; D)=\ndist{(D)}$.

If $\mf(D)$ is a vertex, then we can choose two coprime positive integers $k_1, k_2$ such that $\mf(D)=(\ndist(D)^{-1}, \ndist(D)^{-1})$ is the unique intersection point of $\NP(D)$ with the line $k_2x_1+k_1x_2=(k_1+k_2)\ndist(D)^{-1}$.
Consider the weighted blow-up $\pi: Y\to \mathbb{C}^2$ at $o$ with weight $(k_2,k_1)$.
By Lemma~\ref{lem wt blowup}, we have 
\begin{equation}
 \pi^*(K_{\mathbb{C}^2}+\ndist{(D)}D)
 =K_Y+\pi_*^{-1}(\ndist{(D)}D)+E.\label{the pull back of weighted blow up 0}
\end{equation}
Thus $\lct(\mathbb{C}^2\ni o;D)\leq \ndist{(D)}$.
We claim that $(Y,\pi_*^{-1}(\ndist{(D)}D)+E)$ is lc near $E$. By Lemma~\ref{lem wt blowup}, $$(K_Y+E)|_E=K_E+\left(1-\frac{1}{k_1}\right)P_1+\left(1-\frac{1}{k_2}\right)P_2$$
where $P_1=[0:1]$ and $P_2=[1:0]$ on $E$.
On the other hand, by the choice of $(k_1, k_2)$, $(xy)^{\ndist(D)^{-1}}$ is the unique lowest weight term of $f$, so $$\pi_*^{-1}(D)|_E={\ndist(D)^{-1}}(\frac{1}{k_1}P_1+\frac{1}{k_2}P_2).$$
Hence $$(K_Y+\pi_*^{-1}(\ndist{(D)}D)+E)|_E=K_E+P_1+P_2.$$
By \cite[Theorem~5.50]{KM98}, $(Y, \pi_*^{-1}(\ndist{(D)}D)+E)$ is lc near $E$. By \eqref{the pull back of weighted blow up 0}, $(\mathbb{C}^2\ni o,\ndist{(D)}D)$ is lc, and hence $\lct(\mathbb{C}^2\ni o;D)=\ndist{(D)}$.

Now we may assume that $\mf(D)$ is compact of dimension $1$. Denote its two vertices by $(p_1, q_1)$ and $(p_2, q_2)$ with $p_1<p_2$. Denote $k_1=\frac{p_2-p_1}{\nmul(D)}$ and $k_2=\frac{q_1-q_2}{\nmul(D)}$, then $(k_1, k_2)\in \mathbb{Z}_{>0}^2$ and $\gcd(k_1, k_2)=1$.
Denote 
 $g(x,y)$ to be the sum of all monomial terms of $f$ corresponding to points in $\mf(D)\cap\mathbb{Z}^2$, then
 any monomial appearing in $g$
is of bi-degree 
$
(p_1+{l}k_1, q_1-lk_2)
$
for some integer $0\leq l\leq \nmul(D)$.
We may write $g(x, y)=x^{p_1}y^{q_2}h(x^{k_1}, y^{k_2})$ for some homogeneous polynomial $h\in \mathbb{C}[x, y]$ of degree $\nmul(D)$ with $h(0, 0)\neq 0$. As $\mf(D)$ contains $(\ndist{(D)}^{-1}, \ndist{(D)}^{-1})$, we have $\max\{p_1, q_2\}\leq \ndist{(D)}^{-1}$
and $$k_2p_1+k_1q_1=k_2p_2+k_1q_2=(k_1+k_2)\ndist{(D)}^{-1}.$$
Consider the weighted blow-up $\pi: Y\to \mathbb{C}^2$ at $o$ with weight $(k_2,k_1)$.
By Lemma~\ref{lem wt blowup}, we have 
\begin{equation}
 \pi^*K_{\mathbb{C}^2}
 =K_Y+(1-k_1-k_2)E,\, \pi^*D
 =\pi_*^{-1}D+(k_1+k_2)\ndist{(D)}^{-1}E.\label{the pull back of weighted blow up}
\end{equation}
Thus $\lct(\mathbb{C}^2\ni o;D)\leq \ndist{(D)}$.
Let $r=\min\{\frac{1}{\nmul(D)},\ndist{(D)}\}$, we claim that $(Y,\pi_*^{-1}(rD)+E)$ is lc near $E$. By Lemma~\ref{lem wt blowup}, $(K_Y+\pi_*^{-1}(rD)+E)|_E=K_E+\Delta_E$ with $$\Delta_E=\left(\frac{k_1-1+rp_1}{k_1}\right)P_1+\left(\frac{k_2-1+rq_2}{k_2}\right)P_2+rG,$$
where $P_1=[0:1]$ and $P_2=[1:0]$ on $E$ and $G$
is defined by $(h(z, w)=0)$ in $E\simeq \mathbb{P}^{1}_{z, w}$.
Note that $h$ is of degree $\nmul(D)$, so the coefficients of $G$ are at most $\nmul(D)\leq r^{-1}$. Also note that $\max\{p_1, q_2\}\leq \ndist{(D)}^{-1}\leq r^{-1}$. 
 thus $(E,\Delta_E)$ is lc. By \cite[Theorem~5.50]{KM98}, $(Y, \pi^{-1}_*(rD)+E)$ is lc near $E$. By \eqref{the pull back of weighted blow up}, $\pi^*(K_{\mathbb{C}^2}+r D)\leq K_Y+\pi_*^{-1}(rD)+E$, thus $(\mathbb{C}^2\ni o,rD)$ is lc, and hence $\lct(\mathbb{C}^2\ni o;D)\geq r=\min\{\frac{1}{\nmul(D)},\ndist{(D)}\}$.
\end{proof} 
 
\begin{lem}\label{The case when m fails the inequality}
For any Cartier divisor ${D}$ on $\widehat{\mathbb{C}^2}$ with coordinates $(x,y)$, 
$ \ndist({D})\nmul({D})\leq 2.$ Moreover,
if $\ndist({D})\nmul({D})>1$, then the main face $\mf({D})$ is compact of dimension 1, and either ${\nmul(D)}=p_2-p_1$ or ${\nmul(D)}=q_1-q_2$, where $(p_1, q_1)$ and $(p_2, q_2)$ are two vertices of $\mf({D})$ with $p_1<p_2$.
\end{lem}
\begin{proof}
Suppose that
${D}$ is defined by $(f=0)$ for some $f\in\mathbb{C}[[x,y]]$.
By definition, if $\mf({D})$ is a vertex or not compact, then $ \nmul({D})=\ndist{({D})}^{-1}$. So there is nothing to prove.
Thus we may assume that $\mf({D})$ is compact of dimension $1$ with two vertices $(p_1, q_1)$ and $(p_2, q_2)$ such that $p_1<p_2$. Denote $k_1=\frac{p_2-p_1}{\nmul(D)}$ and $k_2=\frac{q_1-q_2}{\nmul(D)}$, then $(k_1, k_2)\in \mathbb{Z}_{>0}^2$ and $\gcd(k_1, k_2)=1$. Recall that
$$k_2p_1+k_1q_1=k_2p_2+k_1q_2=(k_1+k_2)\ndist{(D)}^{-1}.$$
This implies that 
$$k_2(p_1+p_2)+k_1(q_1+q_2)=\frac{2}{\ndist{(D)}\nmul{(D)}}(p_2-p_1+q_1-q_2).$$
Hence $\ndist{(D)}\nmul{(D)}\leq \frac{2}{\min\{k_1, k_2\}}$. So $\ndist{(D)}\nmul{(D)}\leq {2}$ and if $\ndist{(D)}\nmul{(D)}> 1$ then either $k_1=1$ or $k_2=1$.
\end{proof}

\begin{lem}[{=Corollary~\ref{cor lct B C 1}(a)(c)}]\label{cor lct B C 1 appedix B}
Let ${B}$ be a Cartier divisor in a neighborhood of $o\in \widehat{\mathbb{C}^2}$. 
Suppose that $B$ is irreducible, $\mult_o B=m$. Let $C\neq B$ be a smooth curve passing $o$, and $(B\cdot C)_o=I.$
Let $\lambda$ be a positive real number. Suppose that one of the following condition holds: (a) $\lambda m \leq 1$; or (c) $I\neq m$ and $\lambda I\leq 2$.
Then $(\widehat{\mathbb{C}^2}\ni o, \lambda B)$ is lc and
$$\lct(\widehat{\mathbb{C}^2}\ni o, \lambda B; C)\geq \min\left\{1, 1+\frac{m}{I}-\lambda m\right\}.$$
\end{lem}
\begin{proof}
Note that $I\geq m$ (cf. \cite[\S 1, Excerise 5.4]{GTM52}). So under either condition, $\lambda\leq \min \{1, \frac{1}{m}+\frac{1}{I}\}.$
Denote $t:=\min\{1, 1+\frac{m}{I}-\lambda m\}\geq 0.$ 
It is equivalent to show that $(\widehat{\mathbb{C}^2}\ni o, \lambda B+t C)$ is lc. As being lc is a closed condition on coefficients, we may assume that $t$ is a rational number. 

If $m=I$, then $\lambda m\leq 1$ and \cite[Corollary~5.57]{KM98} implies that $(\widehat{\mathbb{C}^2}\ni o, \lambda B+C)$ is lc, so there is nothing to prove. So we may assume that $I>m$ as $I\ge m$.

Choose local coordinates $(x, y)$ such that $C$ is defined by $(x=0)$.
Suppose that $B$ is defined by $(f=0)$ for some $f\in \mathbb{C}[[x,y]].$ As $B$ is irreducible, by \cite[\S 1, Exercise~5.14]{GTM52}, we may write $f(x,y)=(\alpha x+\beta y)^{m}+g$ for some $(\alpha, \beta)\neq (0,0)$ and $\mult_o g>m$. Note that $I$ is the minimal $k$ such that $y^k$ has non-zero coefficient in $f(x, y)$.
So $I>m$ implies that $\beta=0$. After rescaling, we may assume that 
$f(x, y)=x^m+y^I+h(x, y),$ where $\mult_o h>m$. 

Take a sufficiently divisible positive integer $k$ such that $kt$ and $k\lambda$ are integers, and denote $D=k(\lambda B+tC)$ and $f_D=f(x, y)^{\lambda k}x^{tk}$.
Then our goal is equivalent to show that 
$\lct(\widehat{\mathbb{C}^2}\ni o,D)\geq \frac{1}{k}.$ By Lemma~\ref{the criterion lemma for attaining lct}, it suffices to show that $\ndist{(D)}\geq \frac{1}{k}$ and $\nmul{(D)}\leq {k}$.

First we show that $\ndist{(D)}\geq \frac{1}{k}$. By definition, $\NP(D)$ contains $(tk, \lambda Ik)$ and $(tk+\lambda m k,0)$ as vertices.
By the convexity, $(s, s)\in \NP(D)$ for $s=\frac{(t+\lambda m)Ik}{m+I}$. By the definition of $t$, $s\leq k$. So $\ndist{(D)}\geq \frac{1}{s}\geq \frac{1}{k}$. 

Finally we show that $\nmul{(D)}\leq {k}$. We may assume that $\ndist{({D})}\nmul({D}) >1$.
Then by Lemma~\ref{The case when m fails the inequality},
the main face $\mf({D})$ is compact of dimension 1, and if denote its two vertices by $(p_1, q_1)$ and $(p_2, q_2)$ with $p_1<p_2$, then either ${\nmul(D)}=p_2-p_1$ or ${\nmul(D)}=q_1-q_2$.
Note that we have $tk \leq p_1<p_2\leq tk+\lambda m k$ and $\lambda I k\geq q_1>q_2\geq 0$. Since $\mult_o D=\lambda mk+tk$ and $x^{\lambda mk+tk}$ is the leading term of $f_D$, $\NP(D)$ lies above the line $x_1+x_2=\lambda mk+tk$ with slope $-1$ and intersects this line only at $(tk+\lambda m k,0)$.
So
by the convexity of $\NP(D)$, this implies that the slope of the main face is $\frac{q_1-q_2}{p_1-p_2}<-1.$
So we have ${\nmul(D)}=p_2-p_1$ and $2{\nmul(D)}\leq q_1-q_2$.
This implies that ${\nmul(D)}\leq \min\{p_2-p_1, \frac{1}{2}(q_1-q_2)\}\leq \min\{\lambda m k,\frac{1}{2}\lambda I k \}\leq k.$
\end{proof}

\end{document}